\documentclass[12pt]{article}
\usepackage[latin1]{inputenc}

\usepackage{fullpage}
\usepackage{amsfonts}
\usepackage{amssymb}
\usepackage{amsmath}
\usepackage{color}
\usepackage[mathscr]{euscript}

\usepackage{tikz-cd}

{\normalsize}

{\normalsize}

{\normalsize}

 \newtheorem{theorem}{Theorem}
 \newtheorem{proposition}[theorem]{Proposition}
 \newtheorem{definition}[theorem]{Definition}
 \newtheorem{lemma}[theorem]{Lemma}
 \newtheorem{corollary}[theorem]{Corollary}
 \newtheorem{remark}[theorem]{Remark}
 \newtheorem{assumption}{Assumption}

  \newenvironment{proof}{\small{\bf Proof.}}%
  {\hfill$\Box$\normalsize\bigskip}

\newcommand{\eqsepv}{\; , \enspace}       
\newcommand{\eqfinv}{\; ,}                
\newcommand{\eqfinp}{\; .}

\renewcommand{\bar}{\overline}
\renewcommand{\tilde}{\widetilde}

\newcommand{\mtext}[1]{\,\mbox{#1}\,} %text in maths
\newcommand{\sequence}[2]{\left\{#1\right\}_{#2}}           % Suite
\newcommand{\proscal}[2]{\left\langle#1\:,#2\right\rangle}  % Produit scalaire
\newcommand{\np}[1]{(#1)}                                   % Parenth\`{e}se normal
\newcommand{\bp}[1]{\big(#1\big)}                           % Parenth\`{e}se big
\newcommand{\Bp}[1]{\Big(#1\Big)}                           % Parenth\`{e}se Big
\newcommand{\bgp}[1]{\bigg(#1\bigg)}                        % Parenth\`{e}se bigg
\newcommand{\Bgp}[1]{\Bigg(#1\Bigg)}                        % Parenth\`{e}se Bigg

\newcommand{\bc}[1]{\big[#1\big]}                           % Crochet big
\newcommand{\Bc}[1]{\Big[#1\Big]}                           % Crochet Big
\newcommand{\ba}[1]{\big\{#1\big\}}                         % Accolade big
\newcommand{\RR}{{\mathbb R}} %"ensemble des reels"   
\newcommand{\NN}{{\mathbb N}} %"ensemble des entiers naturels"  
\newcommand{\AAA}{{\mathbb A}} 
\newcommand{\BB}{{\mathbb B}} 
\newcommand{\EE}{{\mathbb E}} %"symbole d'esperance mathematique"
\newcommand{\control}{u}
\newcommand{\Control}{U}
\newcommand{\CONTROL}{{\mathbb U}} 

\newcommand{\state}{x}
\newcommand{\State}{X}
\newcommand{\STATE}{{\mathbb X}}

\newcommand{\uncertain}{w}
\newcommand{\Uncertain}{W}
\newcommand{\UNCERTAIN}{{\mathbb W}} 

\newcommand{\dynamics}{f}
\newcommand{\horizon}{T}

\newcommand{\coutint}{L}                                    % Co\^{u}t int\'{e}gral
\newcommand{\coutfin}{K}                                    % Co\^{u}t final

\newcommand{\Value}{V}

\newcommand{\Hamiltonian}{{\cal H}}

\newcommand{\prbt}{\mathbb{P}}                              % proba  du triplet
\newcommand{\trib}{{\cal F}}

\newcommand{\epro}{(\Omega,\trib,\prbt)}

\newcommand{\va}[1]{\mathbf{#1}}

\def\stackops#1#2#3{%
  \mathrel{\vbox{\offinterlineskip\ialign{%
        \hfil##\hfil\cr
        $#1$\cr
        \noalign{\kern#3}
        $#2$\cr}}}}

\def\plusdot{\stackops{\cdot}{+}{-2.5ex}}

\newcommand{\LowPlus}{\plusdot}  
\newcommand{\UppPlus}{\dotplus}

\newcommand{\PRIMAL}{{\mathbb X}}
\newcommand{\primal}{x}
\newcommand{\DUAL}{{\PRIMAL}^{\sharp}}
\newcommand{\dual}{{\primal}^{\sharp}}
\newcommand{\PRIMALBIS}{{\mathbb Y}}

\newcommand{\primalbis}{y}
\newcommand{\DUALBIS}{{\PRIMALBIS}^{\sharp}}
\newcommand{\dualbis}{{\primalbis}^{\sharp}}
\newcommand{\dualstate}{\state^{\sharp}}
\newcommand{\Dualstate}{\State^{\sharp}}
\newcommand{\DUALSTATE}{\STATE^{\sharp}}
\newcommand{\barRR}{[-\infty,+\infty]} 
\newcommand{\bRR}{[0,+\infty]} 

\newcommand{\dom}{\mathrm{dom}}
\newcommand{\sri}{\mathrm{sri}}
\newcommand{\fonctionprimal}{f} 
\newcommand{\fonctionprimalbis}{g} 
\newcommand{\fonctiontrois}{h} 
\newcommand{\coupling}{c}
\newcommand{\couplingbis}{d}
\newcommand{\couplingter}{\Gamma}
\newcommand{\SumCoupling}[2]{#1 \LowPlus #2} % Product coupling
\newcommand{\kernel}{{\cal K}} 
\newcommand{\convolution}{{\cal I}} 
\newcommand{\tildeconvolution}{\overline{\convolution}}
\newcommand{\tildetildeconvolution}{\underline{\convolution}}
\newcommand{\Exchange}{{\cal E}}
\newcommand{\LFM}[1]{#1^{\star}}
\newcommand{\LFMbi}[1]{#1^{\star\star}}
\newcommand{\SFM}[2]{#1^{#2}}
\newcommand{\SFMbi}[2]{#1^{#2{#2}'}}
\newcommand{\espaceL}[1]{\mathbb{L}^{#1}\bp{\epro,\RR^{n_{\STATE}}}}

\title{Fenchel-Moreau Conjugation Inequalities\\
with Three Couplings\\
  and Application to Stochastic Bellman Equation}

\author{Jean-Philippe Chancelier and Michel De Lara\footnote{delara@cermics.enpc.fr}
\\ Universit\'{e} Paris-Est, CERMICS (ENPC) }

\begin{document}

\maketitle

\begin{abstract}
  Given two couplings between ``primal'' and ``dual'' sets, 
  we prove a general implication 
  that relates an inequality involving ``primal'' sets 
  to a reverse inequality involving the ``dual'' sets.
  More precisely,
  let be given two ``primal'' sets $\PRIMAL$, $\PRIMALBIS$
and two ``dual'' sets $\DUAL$, $\DUALBIS$, 
together with two {coupling} functions
  \( \PRIMAL \overset{\coupling}{\leftrightarrow} \DUAL \) and
  \( \PRIMALBIS \overset{\couplingbis}{\leftrightarrow} \DUALBIS \).
  We define a new coupling \( \SumCoupling{\coupling}{\couplingbis} \)
  between the ``primal'' product set~$\PRIMAL \times \PRIMALBIS$ 
  and the ``dual'' product set $\DUAL \times \DUALBIS$.
  Then, we consider any bivariate function 
  \( \kernel : \PRIMAL \times \PRIMALBIS \to \barRR \)
  and univariate functions 
  \( \fonctionprimal : \PRIMAL  \to \barRR \) and \( \fonctionprimalbis : \PRIMALBIS  \to \barRR \), 
  all defined on the ``primal'' sets. 
  We prove that \( \fonctionprimal\np{\primal} \geq \inf_{\primalbis \in \PRIMALBIS} 
  \Bp{ \kernel\np{\primal, \primalbis} \UppPlus 
    \fonctionprimalbis\np{\primalbis} } \)
  \(  \Rightarrow 
  \SFM{\fonctionprimal}{\coupling}\np{\dual} \leq \inf_{\dualbis \in \DUALBIS} 
  \Bp{ \SFM{\kernel}{\SumCoupling{\coupling}{\couplingbis}}\np{\dual,\dualbis} 
    \UppPlus \SFM{\fonctionprimalbis}{-\couplingbis}\np{\dualbis} } \),
  where we stress that the Fenchel-Moreau conjugates 
  \( \SFM{\fonctionprimal}{\coupling} \) and 
  \( \SFM{\fonctionprimalbis}{-\couplingbis}\) 
  are not necessarily taken with the same coupling. 
We study the equality case. 
  We display several  applications.
  We provide a new formula for the Fenchel-Moreau conjugate of a
  generalized inf-convolution.
  We obtain formulas with partial Fenchel-Moreau conjugates.
  Finally, we consider the Bellman equation in stochastic dynamic 
  programming and we  provide a ``Bellman-like'' 
  equation for the Fenchel conjugates of the value functions.
\end{abstract}

% \pagebreak
% \tableofcontents
% \pagebreak

\section{Introduction}

In convex analysis, the Fenchel conjugacy plays a central role.
It is involved in many equalities and inequalities, 
like the well known Fenchel (in)equalities
or the Fenchel conjugate of an inf-convolution.
The classical Fenchel conjugate was extended by J.~J. Moreau~\cite{Moreau:1970}, 
by replacing the bilinear pairing,
between a vector space and its dual,
with a more general coupling. 
This gives the so-called Fenchel-Moreau conjugate
(see Chapter~11L and the Commentary in 
\cite{Rockafellar-Wets:1998} with a brief historical perspective
and references).
In abstract convexity~\cite{MR1834382,MR1461544,MR1442257,Martinez-Legaz:2005}, 
affine functions are replaced by another class
of functions (related to the coupling), and so are convex functions
(replaced by so-called abstract convex functions), by taking the supremum. 
In this way, generalized Fenchel conjugation formulas are obtained, 
as well as duality for abstract convex functions.
Generalized Fenchel conjugation also appears in the dual formulation 
of optimal transport problems~\cite{MR2279394,Santambrogio:2015}. 
Calculus with different couplings can be found in~\cite{Cabot-Jourani-Thibault:2018}.

In this paper, we provide a main Fenchel-Moreau conjugation 
inequality with three couplings, and applications.
In Sect.~\ref{A_duality_formula_with_multiple_Fenchel-Moreau_conjugates},
we establish our main inequality. Then, we provide sufficient conditions
for the equality case. 
In Sect.~\ref{Applications}, we display several  applications.
First, we provide a definition of a generalized inf-convolution,
and new formulas for its Fenchel-Moreau conjugate (inequality and equality).
Second, we obtain formulas with partial Fenchel-Moreau conjugates.
Finally, we consider the Bellman equation in stochastic dynamic 
programming and we  provide a ``Bellman-like'' 
equation for the Fenchel conjugates of the value functions.

\section{Duality inequality with three Fenchel-Moreau conjugates}
\label{A_duality_formula_with_multiple_Fenchel-Moreau_conjugates}

Given two couplings between ``primal'' and ``dual'' sets, 
we prove a general implication 
that relates an inequality involving ``primal'' sets 
to a reverse inequality involving the ``dual'' sets.

In what follows, we rely upon 
background on J.~J. Moreau lower and upper additions and 
on Fenchel-Moreau conjugacy with respect to a coupling, that 
can be found in Appendix~\ref{Appendix}.

\subsection{Main duality inequality}

Let be given two ``primal'' sets $\PRIMAL$, $\PRIMALBIS$
and two ``dual'' sets $\DUAL$, $\DUALBIS$, 
together with two \emph{coupling} functions
\begin{equation}
  \coupling : \PRIMAL \times \DUAL \to \barRR \eqsepv
  \couplingbis : \PRIMALBIS \times \DUALBIS \to \barRR \eqfinp 
\end{equation}
We will call $\PRIMAL$ and $\PRIMALBIS$ ``primal'' sets,
whereas $\DUAL$ and $\DUALBIS$ are ``dual'' sets.

We define the \emph{sum coupling} \( \SumCoupling{\coupling}{\couplingbis} \) 
--- coupling the ``primal'' product set~$\PRIMAL \times \PRIMALBIS$ 
with the ``dual'' product set $\DUAL \times \DUALBIS$ ---
by
\begin{subequations}
  \begin{align}
    \SumCoupling{\coupling}{\couplingbis}: 
    \bp{\PRIMAL \times \PRIMALBIS} \times 
    \bp{\DUAL \times \DUALBIS} 
    & \to\barRR \eqfinv \\
    \bp{ \np{\primal,\primalbis} ,\np{\dual,\dualbis}}\hspace{0.6cm}
    &\mapsto 
      \coupling\np{\primal,\dual} \LowPlus 
      \couplingbis\np{\primalbis,\dualbis} 
      \eqfinp 
  \end{align}
  \label{eq:product_coupling}
\end{subequations}
With any bivariate function 
\( \kernel : \PRIMAL \times \PRIMALBIS \to \barRR \), 
defined on the ``primal'' product set~$\PRIMAL \times \PRIMALBIS$, 
we associate the conjugate, with respect to the 
coupling \( \SumCoupling{\coupling}{\couplingbis} \), defined 
on the ``dual'' product set $\DUAL \times \DUALBIS$, by:
\begin{subequations}
  \begin{align}
    \SFM{\kernel}{\SumCoupling{\coupling}{\couplingbis}}\np{\dual,\dualbis} 
    &=     \sup_{\primal \in \PRIMAL, \primalbis \in \PRIMALBIS} 
      \Bp{ \bp{\SumCoupling{\coupling}{\couplingbis}}%
      \bp{\np{\primal,\primalbis},\np{\dual,\dualbis}} 
      \LowPlus \bp{ - \kernel \np{\primal,\primalbis}} } \\
    &=\sup_{\primal \in \PRIMAL, \primalbis \in \PRIMALBIS} 
      \Bp{ 
      \coupling\np{\primal,\dual} \LowPlus 
      \couplingbis\np{\primalbis,\dualbis}
      \LowPlus \bp{ - \kernel \np{\primal,\primalbis}}} \\
    & \qquad \forall \np{\dual,\dualbis} \in \DUAL \times \DUALBIS \eqfinp \nonumber
  \end{align}
  \label{eq:product_conjugate}
\end{subequations}

\begin{figure}
  \centering
  \[
  \begin{tikzcd}[column sep=huge, labels={font=\everymath\expandafter{\the\everymath\textstyle}}]
    \PRIMAL \arrow[r, " \coupling " ]  \arrow{d}[description]{\quad} 
    & \DUAL \arrow[l] \arrow{d}[description]{\quad} \\[1.5cm]
    \PRIMALBIS \arrow[r, "\couplingbis " ] \arrow{u}[description]{\kernel} 
    & \DUALBIS \arrow[l] \arrow{u}[description]{\SFM{\kernel}{\SumCoupling{\coupling}{\couplingbis}}}
  \end{tikzcd}
  \]
  \caption{A kernel~$\kernel$, two couplings $\coupling$ and $\couplingbis$, and a
    new kernel~$\SFM{\kernel}{\protect\SumCoupling{\coupling}{\couplingbis}}$}
\end{figure}

% \[
% \begin{tikzcd}[column sep=huge, labels={font=\everymath\expandafter{\the\everymath\textstyle}}]
%   \extendedR^\PRIMAL \arrow[r, " \fonctionprimal \mapsto \fonctionprimal^\coupling " ] 
%   & \extendedR^{\DUAL} \\[1.5cm]
%   \extendedR^\PRIMALBIS \arrow{u}[description]{\fonctionprimalbis \mapsto \fonctionprimalbis^\kernel} \arrow[r, "\fonctionprimalbis \mapsto \fonctionprimalbis^\couplingbis " ] % \arrow[ur]
%   & \extendedR^{\DUALBIS} \arrow{u}[description]{\fonctionprimalbis \mapsto \fonctionprimalbis_{\tilde\kernel}}
% \end{tikzcd}
% \]

In what follows, we will call the function~$\kernel$ a \emph{kernel}
(or a \emph{potential}). Indeed, consider the expression in 
the left hand side assumption in~\eqref{eq:main}.
If we translate it from the $(\min,+)$ algebra
to the usual $(+,\times)$ algebra, it 
stands as an integration with respect to a kernel.

\begin{theorem}
  For any bivariate function 
  \( \kernel : \PRIMAL \times \PRIMALBIS \to \barRR \)
  and univariate functions 
  \( \fonctionprimal : \PRIMAL  \to \barRR \) and \( \fonctionprimalbis : \PRIMALBIS  \to \barRR \), 
  all defined on the ``primal'' sets, 
  we have that
  \begin{multline}
    \fonctionprimal\np{\primal} \geq \inf_{\primalbis \in \PRIMALBIS} 
    \Bp{ \kernel\np{\primal, \primalbis} \UppPlus \fonctionprimalbis\np{\primalbis} } 
    \eqsepv \forall \primal \in \PRIMAL 
    \Rightarrow \\
    \SFM{\fonctionprimal}{\coupling}\np{\dual} \leq \inf_{\dualbis \in \DUALBIS} 
    \Bp{ \SFM{\kernel}{\SumCoupling{\coupling}{\couplingbis}}%
      \np{\dual,\dualbis} 
      \UppPlus \SFM{\fonctionprimalbis}{-\couplingbis}\np{\dualbis} } 
    \eqsepv \forall \dual \in \DUAL 
    \eqfinp
    \label{eq:main}
  \end{multline}
  \label{th:main}
\end{theorem}
Notice that the left hand side assumption in~\eqref{eq:main} is
a rather weak inequality (upper bound for an infimum), whereas 
the right hand side assumption in~\eqref{eq:main} is
a rather strong inequality (lower bound for an infimum).
\bigskip

\begin{proof}
  \begin{subequations}
    \begin{align}
      \SFM{\fonctionprimal}{\coupling}\np{\dual} 
      &= 
        \sup_{\primal \in \PRIMAL} \Bp{ \coupling\np{\primal,\dual} 
        \LowPlus \bp{-\fonctionprimal\np{\primal}} } 
        \nonumber 
        \intertext{ by definition~\eqref{eq:upper-Fenchel-Moreau_conjugate} 
        of the conjugate~$\SFM{\fonctionprimal}{\coupling}\np{\dual}$ }
      &\leq 
        \sup_{\primal \in \PRIMAL} \bgp{ \coupling\np{\primal,\dual} 
        \LowPlus \Bp{-
        \inf_{\primalbis \in \PRIMALBIS} \bp{ \kernel\np{\primal,\primalbis} 
        \UppPlus \fonctionprimalbis\np{\primalbis} } } } 
        \nonumber 
        \intertext{ by the left hand side assumption in~\eqref{eq:main} 
        and by the property~\eqref{eq:lower_addition_leq} that 
        the operator~$\LowPlus$ is monotone 
        \protect\newline[this inequality is an equality when 
        the left hand side assumption in~\eqref{eq:main} is an equality]}
      &= 
        \sup_{\primal \in \PRIMAL} \bgp{ \coupling\np{\primal,\dual} 
        \LowPlus
        \sup_{\primalbis \in \PRIMALBIS} 
        \Bp{ - \bp{ \kernel\np{\primal,\primalbis} \UppPlus 
        \fonctionprimalbis\np{\primalbis} } } }
        \nonumber 
        \intertext{ by \(-\inf = \sup - \)}
      &=
        \sup_{\primal \in \PRIMAL, \primalbis \in \PRIMALBIS} 
        \bgp{ \coupling\np{\primal,\dual} \LowPlus  
        \Bp{ - \bp{ \kernel\np{\primal,\primalbis} \UppPlus 
        \fonctionprimalbis\np{\primalbis} } } } 
        \nonumber 
        \intertext{ by the property~\eqref{eq:lower_addition_sup} 
that the operator~$\sup$ is ``distributive'' with respect to~$\LowPlus$ }
      &\leq 
        \sup_{\primal \in \PRIMAL, \primalbis \in \PRIMALBIS} 
        \Bp{ \coupling\np{\primal,\dual} \LowPlus  
        \Bp{ - \bp{ \kernel\np{\primal,\primalbis} \UppPlus 
        \SFMbi{\fonctionprimalbis}{(-\couplingbis)}\np{\primalbis} } } } 
        \nonumber 
        \intertext{ because 
        \( \SFMbi{\fonctionprimalbis}{(-\couplingbis)} \leq \fonctionprimalbis \) 
        by~\eqref{eq:Fenchel-Moreau_biconjugate_inequality} 
        and by the property~\eqref{eq:lower_addition_leq} that 
        the operator~$\LowPlus$ is monotone 
        \protect\newline[this inequality is an equality when 
        \( \SFMbi{\fonctionprimalbis}{(-\couplingbis)} = \fonctionprimalbis \) ]}
      &=
        \sup_{\primal \in \PRIMAL, \primalbis \in \PRIMALBIS} 
        \Bgp{ \coupling\np{\primal,\dual} \LowPlus 
        \Bgp{ - \bgp{ \kernel\np{\primal,\primalbis} \UppPlus 
        \Bp{ \sup_{\dualbis \in \DUALBIS} 
        \Bp{ 
        \bp{- \couplingbis\np{\primalbis,\dualbis} } 
        \LowPlus \bp{ -\SFM{\fonctionprimalbis}{-\couplingbis}\np{\dualbis} } } } }
        } } 
        \nonumber 
        \intertext{ by definition~\eqref{eq:Fenchel-Moreau_biconjugate} of the 
        biconjugate~\( \SFMbi{\fonctionprimalbis}{(-\couplingbis)} \) }
      &\leq 
        \sup_{\primal \in \PRIMAL, \primalbis \in \PRIMALBIS} 
        \Bgp{ \coupling\np{\primal,\dual} \LowPlus  
        \bgp{ - \sup_{\dualbis \in \DUALBIS} 
        \Bp{ \kernel\np{\primal,\primalbis} \UppPlus 
        \bp{- \couplingbis\np{\primalbis,\dualbis}
        \LowPlus \bp{ -\SFM{\fonctionprimalbis}{-\couplingbis}\np{\dualbis} } } } }
        } 
        \nonumber 
        \intertext{ by the property~\eqref{eq:upper_addition_sup} that the 
        operator~$\sup$ is "subdistributive'' with respect to~$\UppPlus$,
        and by the property~\eqref{eq:upper_addition_leq} that 
        the operator~$\UppPlus$ is monotone 
        \protect\newline[this inequality is an equality when 
        \(-\infty < \kernel\np{\primal,\primalbis} \) 
        by~\eqref{eq:upper_addition_sup_constant}]}
      &=
        \sup_{\primal \in \PRIMAL, \primalbis \in \PRIMALBIS} 
        \Bgp{ \coupling\np{\primal,\dual} \LowPlus 
        \inf_{\dualbis \in \DUALBIS} -\bgp{ \kernel\np{\primal,\primalbis} 
        \UppPlus 
        \Bp{\bp{- \couplingbis\np{\primalbis,\dualbis}} 
        \LowPlus \bp{ -\SFM{\fonctionprimalbis}{-\couplingbis}\np{\dualbis} } } } } 
        \nonumber 
        \intertext{ by \( -\sup = \inf - \)}
      &\leq 
        \sup_{\primal \in \PRIMAL, \primalbis \in \PRIMALBIS} 
        \inf_{\dualbis \in \DUALBIS} 
        \Bgp{ \coupling\np{\primal,\dual} \LowPlus 
        \Bgp{ -\bgp{ \kernel\np{\primal,\primalbis} \UppPlus 
        \Bp{\bp{- \couplingbis\np{\primalbis,\dualbis}} 
        \LowPlus \bp{ -\SFM{\fonctionprimalbis}{-\couplingbis}\np{\dualbis} } } } } }
        \nonumber 
        \intertext{ by the property~\eqref{eq:lower_addition_inf} that the 
        operator~$\inf$ is "subdistributive'' with respect to~$\LowPlus$
        \protect\newline[this inequality is an equality when 
        \( \coupling\np{\primal,\dual} < +\infty \) 
        by~\eqref{eq:lower_addition_inf_constant}]}
      &= 
        \sup_{\primal \in \PRIMAL, \primalbis \in \PRIMALBIS} 
        \inf_{\dualbis \in \DUALBIS} 
        \Bgp{ \coupling\np{\primal,\dual} \LowPlus 
        \bp{ -\kernel\np{\primal,\primalbis} } 
        \LowPlus \bgp{ -\Bp{\bp{- \couplingbis\np{\primalbis,\dualbis}} 
        \LowPlus 
        \bp{-\SFM{\fonctionprimalbis}{-\couplingbis}\np{\dualbis}} } } } 
        \nonumber 
        \intertext{ by the correspondence~\eqref{eq:lower_upper_addition_minus} 
        between $\LowPlus$ and $\UppPlus$ by means of $a \mapsto -a$}
      &= 
        \sup_{\primal \in \PRIMAL, \primalbis \in \PRIMALBIS} 
        \inf_{\dualbis \in \DUALBIS} 
        \Bgp{ \coupling\np{\primal,\dual} \LowPlus 
        \bp{ -\kernel\np{\primal,\primalbis} } 
        \LowPlus \Bp{\couplingbis\np{\primalbis,\dualbis}
        \UppPlus \SFM{\fonctionprimalbis}{-\couplingbis}\np{\dualbis} } }  
        \nonumber 
        \intertext{ by the correspondence~\eqref{eq:lower_upper_addition_minus} 
        between $\LowPlus$ and $\UppPlus$ by means of $a \mapsto -a$}
      &= 
        \sup_{\primal \in \PRIMAL, \primalbis \in \PRIMALBIS} 
        \inf_{\dualbis \in \DUALBIS} 
        \Bgp{
        \bgp{\coupling\np{\primal,\dual} \LowPlus 
        \bp{ -\kernel\np{\primal,\primalbis} }}
        \LowPlus
        \Bp{ \couplingbis\np{\primalbis,\dualbis} 
        \UppPlus \SFM{\fonctionprimalbis}{-\couplingbis}\np{\dualbis} } } 
        \nonumber 
        \intertext{by associativity of~$\LowPlus$ }
      &\leq \sup_{\primal \in \PRIMAL, \primalbis \in \PRIMALBIS} 
        \inf_{\dualbis \in \DUALBIS} 
        \Bgp{
        \bgp{ \Bp{\coupling\np{\primal,\dual} \LowPlus 
        \bp{ -\kernel\np{\primal,\primalbis} }} 
        \LowPlus
        \couplingbis\np{\primalbis,\dualbis}} 
        \UppPlus \SFM{\fonctionprimalbis}{-\couplingbis}\np{\dualbis} } 
        \nonumber 
        \intertext{by the inequality~\eqref{eq:lower_upper_addition_inequality}
        \protect\newline[this inequality is an equality when 
        \( -\infty <\coupling\np{\primal,\dual} \LowPlus
        \bp{ -\kernel\np{\primal,\primalbis} } < +\infty \) and 
        \( -\infty <\couplingbis\np{\primalbis,\dualbis} < +\infty \)]
        } 
        % u=\SFM{\fonctionprimalbis}{-\couplingbis}\np{\dualbis} 
        % v=\couplingbis\np{\primalbis,\dualbis} 
        % w=\coupling\np{\primal,\dual} \LowPlus \bp{ -\kernel\np{\primal,\primalbis} }
        % 
      &= 
        \sup_{\primal \in \PRIMAL, \primalbis \in \PRIMALBIS} 
        \inf_{\dualbis \in \DUALBIS} 
        \bgp{
        \Bp{\coupling\np{\primal,\dual} \LowPlus 
        \bp{ -\kernel\np{\primal,\primalbis} }
        \LowPlus 
        \couplingbis\np{\primalbis,\dualbis}}
        \UppPlus \SFM{\fonctionprimalbis}{-\couplingbis}\np{\dualbis}} 
        \label{eq:main_supinf} 
        \intertext{by associativity of~$\LowPlus$}
      &\leq 
        \inf_{\dualbis \in \DUALBIS} 
        \sup_{\primal \in \PRIMAL, \primalbis \in \PRIMALBIS} 
        \bgp{
        \Bp{\coupling\np{\primal,\dual} \LowPlus 
        \bp{ -\kernel\np{\primal,\primalbis} }
        \LowPlus 
        \couplingbis\np{\primalbis,\dualbis}}
        \UppPlus \SFM{\fonctionprimalbis}{-\couplingbis}\np{\dualbis}}  
        \label{eq:main_infsup}
        \intertext{ by \( \sup\inf \leq \inf\sup \) }
      &= 
        \inf_{\dualbis \in \DUALBIS} 
        \Bgp{ \sup_{\primal \in \PRIMAL, \primalbis \in \PRIMALBIS} 
        \Bp{\coupling\np{\primal,\dual} \LowPlus 
        \bp{ -\kernel\np{\primal,\primalbis} }
        \LowPlus 
        \couplingbis\np{\primalbis,\dualbis}} 
        \UppPlus \SFM{\fonctionprimalbis}{-\couplingbis}\np{\dualbis}}  
        \nonumber 
        \intertext{by the property~\eqref{eq:lower_addition_sup} that the 
        operator~$\sup$ is "subdistributive'' with respect to~$\LowPlus$}
      &=
        \inf_{\dualbis \in \DUALBIS} \Bp{ 
        \SFM{\kernel}{\SumCoupling{\coupling}{\couplingbis}}\np{\dual,\dualbis}
        \UppPlus \SFM{\fonctionprimalbis}{-\couplingbis}\np{\dualbis} } 
        \nonumber \eqfinp
    \end{align}
  \end{subequations}
by the definition~\eqref{eq:product_conjugate}
        of~$\SFM{\kernel}{\SumCoupling{\coupling}{\couplingbis}}$.

  This ends the proof. 
\end{proof}

\subsection{The duality equality case}
\label{The_duality_equality_case}

% Let be given two {coupling} functions
% \( \coupling : \PRIMAL \times \DUAL \to \barRR \) and
% \( \couplingbis : \PRIMALBIS \times \DUALBIS \to \barRR \).

The equality case in~\eqref{eq:main} is the property that 
\begin{multline}
  \fonctionprimal\np{\primal} = \inf_{\primalbis \in \PRIMALBIS} 
  \Bp{ \kernel\np{\primal,\primalbis} \UppPlus 
    \fonctionprimalbis\np{\primalbis} } 
  \eqsepv \forall \primal \in \PRIMAL 
  \Rightarrow \\
  \SFM{\fonctionprimal}{\coupling}\np{\dual} = \inf_{\dualbis \in \DUALBIS} 
  \Bp{ \SFM{\kernel}{\SumCoupling{\coupling}{\couplingbis}}%
    \np{\dual,\dualbis} 
    \UppPlus \SFM{\fonctionprimalbis}{-\couplingbis}\np{\dualbis} } 
  \eqsepv \forall \dual \in \DUAL 
  \eqfinp
  \label{eq:main==}
\end{multline}
We will now provide sufficient conditions under which the
equality case~\eqref{eq:main==} holds true in different cases:
with real-valued couplings and real-valued kernel;
with extended couplings and extended kernel;
with one bilinear coupling and extended kernel.

% In \cite{Mazure-Volle:1991}, the equation 
% \( \fonctionprimal\np{\primal} = \inf_{\primalbis \in \PRIMALBIS} 
%   \Bp{ \kernel\np{\primal,\primalbis} \UppPlus 
%     \fonctionprimalbis\np{\primalbis} } \) 
% with unknown function~\( \fonctionprimalbis \) is studied.

\subsubsection{With real-valued couplings and real-valued kernel}

We consider the case where
both the couplings and the kernel take real values,
whereas all the other functions can take extended values. 

\begin{corollary}
  Consider any bivariate function 
  \( \kernel : \PRIMAL \times \PRIMALBIS \to \barRR \)
and univariate functions 
  \( \fonctionprimal : \PRIMAL  \to \barRR \) and 
  \( \fonctionprimalbis : \PRIMALBIS  \to \barRR \), 
  all defined on the ``primal'' sets.
  Suppose that
  \begin{enumerate}
  \item 
    \( \SFMbi{\fonctionprimalbis}{(-\couplingbis)} = \fonctionprimalbis \);
  \item 
    \label{it:Sion}
    % the equality between~\eqref{eq:main_supinf}
    % and \eqref{eq:main_infsup} holds true, that is 
we have strong duality 
 %   \begin{multline}
\begin{align}
&\sup_{\primal \in \PRIMAL, \primalbis \in \PRIMALBIS} 
        \inf_{\dualbis \in \DUALBIS} 
        \bgp{
        \Bp{\coupling\np{\primal,\dual} \LowPlus 
        \bp{ -\kernel\np{\primal,\primalbis} }
        \LowPlus 
        \couplingbis\np{\primalbis,\dualbis}}
        \UppPlus \SFM{\fonctionprimalbis}{-\couplingbis}\np{\dualbis}} 
\nonumber \\
 = &
        \inf_{\dualbis \in \DUALBIS} 
        \sup_{\primal \in \PRIMAL, \primalbis \in \PRIMALBIS} 
        \bgp{
        \Bp{\coupling\np{\primal,\dual} \LowPlus 
        \bp{ -\kernel\np{\primal,\primalbis} }
        \LowPlus 
        \couplingbis\np{\primalbis,\dualbis}}
        \UppPlus \SFM{\fonctionprimalbis}{-\couplingbis}\np{\dualbis}}  
     \eqfinv 
\end{align}
%    \end{multline}
for all \( \dual \in \DUAL \);
% as, for instance, when
%     % \begin{enumerate}
%     % \item 
%     %   the sets $\PRIMAL$, $\PRIMALBIS$ and $\DUALBIS$ are convex compact,
%     % \item 
%     the following function 
%     % is quasi-concave-convex u.s.c.-l.s.c. 
%     has a saddle point (or has no duality gap) for all \( \dual \in \DUAL \) 
%     % \end{enumerate}
%     \begin{equation}
%       \bp{ \np{\primal, \primalbis} , \dualbis } 
%       \in \np{\PRIMAL \times \PRIMALBIS} \times \DUALBIS \mapsto 
%       \Bp{\coupling\np{\primal,\dual} \LowPlus \bp{ -\kernel\np{\primal,\primalbis} }
%         \LowPlus 
%         \couplingbis\np{\primalbis,\dualbis} }
%       \UppPlus \SFM{\fonctionprimalbis}{-\couplingbis}\np{\dualbis} 
% \eqfinv
%     \end{equation}
  \item 
the two {coupling} functions
\( \coupling : \PRIMAL \times \DUAL \to \RR \) and
\( \couplingbis : \PRIMALBIS \times \DUALBIS \to \RR \),
and the kernel 
  \( \kernel : \PRIMAL \times \PRIMALBIS \to \RR \)
all take finite values.
  \end{enumerate}
  Then, the equality case~\eqref{eq:main==} holds true.
\label{cor:main==real-valued_couplings}
\end{corollary}

\begin{proof}
  Following the proof of Theorem~\ref{th:main}, all but one inequality 
--- namely \( \sup\inf \leq \inf\sup \) between~\eqref{eq:main_supinf}
  and \eqref{eq:main_infsup} --- become equalities when the functions
  \( \coupling : \PRIMAL \times \DUAL \to \RR \),
  \( \couplingbis : \PRIMALBIS \times \DUALBIS \to \RR \)
  and \( \kernel : \PRIMAL \times \PRIMALBIS \to \RR \)
  take real values 
  and when \( \SFMbi{\fonctionprimalbis}{(-\couplingbis)} 
  = \fonctionprimalbis \). 
  Once we have the 
  % The conditions in item~\ref{it:Sion} imply the 
  equality between~\eqref{eq:main_supinf}
  and \eqref{eq:main_infsup}, we obtain that 
  the equality case~\eqref{eq:main==} holds true.
  % by Sion's Theorem.
\end{proof}

\subsubsection{With extended couplings and extended kernel}

We consider the case where the couplings, the kernel and 
all the other functions can take extended values. 

\begin{corollary}
  Consider any bivariate function 
  \( \kernel : \PRIMAL \times \PRIMALBIS \to \barRR \)
and univariate functions 
  \( \fonctionprimal : \PRIMAL  \to \barRR \) and 
  \( \fonctionprimalbis : \PRIMALBIS  \to \barRR \), 
  all defined on the ``primal'' sets.
  We define
  \begin{equation}
    \kernel_{\dual}  \np{\primalbis} =
- \Bp{\SFM{ \kernel\np{\cdot,\primalbis} }{\coupling}\np{\dual}} =
    \inf_{\primal \in \PRIMAL} 
    \Bp{ \bp{-\coupling\np{\primal,\dual}} \UppPlus 
      \kernel\np{\primal, \primalbis }}
    \eqsepv \forall \np{\dual,\primalbis} \in \DUAL\times\PRIMALBIS
    \eqfinp
    \label{eq:marginal_kernel}
  \end{equation}
  Suppose that
  \begin{equation}
\sup_{ \primalbis \in \PRIMALBIS} \bgp{ 
       \bp{ -\kernel_{\dual} \np{\primalbis} }  \LowPlus 
\bp{ -\fonctionprimalbis\np{\primalbis} } }
=
        \inf_{ \dualbis \in \DUALBIS} \bgp{
        \kernel_{\dual}^{\couplingbis}\np{\dualbis}
        \UppPlus \SFM{\fonctionprimalbis}{-\couplingbis}\np{\dualbis} } 
\eqfinp
\label{eq:rfd-post}
  \end{equation}
  Then, the equality case~\eqref{eq:main==} holds true.
\label{cor:main==extended_couplings}
\end{corollary}

\begin{proof}
   First, to prove the equality result~\eqref{eq:main==}, 
  we start by stating the so-called Fenchel inequality,
but with a general coupling:
  for any two 
  functions \( \fonctiontrois : \PRIMALBIS  \to \barRR \) 
  and \( \fonctionprimalbis : \PRIMALBIS  \to \barRR \), 
  we have that
  \begin{equation}
      \sup_{\primalbis \in \PRIMALBIS} \Bp{ \bp{-\fonctiontrois\np{\primalbis}}
      \LowPlus \bp{-\fonctionprimalbis\np{\primalbis}} } 
      \leq
        \inf_{\dualbis \in \DUALBIS} 
        \Bp{ \SFM{\fonctiontrois}{\couplingbis}\np{\dualbis} 
        \UppPlus \SFM{\fonctionprimalbis}{-\couplingbis}\np{\dualbis} } 
        \eqfinp
        \label{eq:Fenchel_inequality_with_general_coupling_un}     
  \end{equation}
The proof easily follows from the definition of the Fenchel-Moreau conjugate
in~\eqref{eq:upper-Fenchel-Moreau_conjugate}. 
It is also a corollary of Theorem~\ref{th:main} when 
  we take singleton sets $\PRIMAL=\{ \primal \}$ and
$\DUAL=\{ \dual \}$, with the null coupling
\( \coupling\np{\primal,\dual}=0 \).
\bigskip

Second, we 
give a new proof of~\eqref{eq:main} in Theorem~\ref{th:main}.
  \begin{subequations}
    \begin{align}
      \SFM{\fonctionprimal}{\coupling}\np{\dual} 
      &= 
        \sup_{\primal \in \PRIMAL} 
\Bp{ \coupling\np{\primal,\dual} \LowPlus \bp{-\fonctionprimal\np{\primal}} } 
\nonumber 
      \intertext{ by definition~\eqref{eq:upper-Fenchel-Moreau_conjugate} 
      of the conjugate~$\SFM{\fonctionprimal}{\coupling}\np{\dual}$ }
      &\leq 
\sup_{\primal \in \PRIMAL} \bgp{ \coupling\np{\primal,\dual} \LowPlus \Bp{-
        \inf_{\primalbis \in \PRIMALBIS} \bp{ \kernel\np{\primal,\primalbis} 
\UppPlus \fonctionprimalbis\np{\primalbis} } } }
        \label{ineq:un}
      \intertext{ by the left hand side assumption in~\eqref{eq:main} 
      and by the property~\eqref{eq:lower_addition_leq} that 
      the operator~$\LowPlus$ is monotone }
      &= 
\sup_{\primal \in \PRIMAL} \bgp{ \coupling\np{\primal,\dual} \LowPlus
        \sup_{\primalbis \in \PRIMALBIS} 
\Bp{ - \bp{ \kernel\np{\primal,\primalbis} \UppPlus 
\fonctionprimalbis\np{\primalbis} } } }
\nonumber
      \intertext{ by \(-\inf = \sup - \)}
      &=
\sup_{\primal \in \PRIMAL, \primalbis \in \PRIMALBIS} 
\bgp{ \coupling\np{\primal,\dual} \LowPlus  
        \Bp{ - \bp{ \kernel\np{\primal,\primalbis} \UppPlus 
\fonctionprimalbis\np{\primalbis} } } }
\nonumber 
      \intertext{ by the property~\eqref{eq:lower_addition_sup} that the 
operator~$\sup$ is ``distributive'' with respect to~$\LowPlus$ }
      &=
\sup_{\primal \in \PRIMAL, \primalbis \in \PRIMALBIS} 
\bgp{ \Bp{ \coupling\np{\primal,\dual} \LowPlus  
        \bp{ - \kernel\np{\primal,\primalbis}}} \LowPlus 
\bp{ -\fonctionprimalbis\np{\primalbis} } } 
\nonumber 
      \intertext{ by \eqref{eq:lower_upper_addition_minus} 
and by associativity of~$\LowPlus$ } 
      &=
\sup_{ \primalbis \in \PRIMALBIS} \bgp{ 
        \sup_{\primal \in \PRIMAL} \Bp{ \coupling\np{\primal,\dual} \LowPlus  
        \bp{ - \kernel\np{\primal,\primalbis}} } \LowPlus 
\bp{ -\fonctionprimalbis\np{\primalbis} } }
\nonumber 
      \intertext{ by the property~\eqref{eq:lower_addition_sup} that the 
operator~$\sup$ is ``distributive'' with respect to~$\LowPlus$ }
      &=
\sup_{ \primalbis \in \PRIMALBIS} \bgp{ 
       \bp{ -\kernel_{\dual} \np{\primalbis} }  \LowPlus 
\bp{ -\fonctionprimalbis\np{\primalbis} } }
\nonumber 
      \intertext{as \( \kernel_{\dual} \np{\primalbis} =
-\sup_{\primal \in \PRIMAL} \Bp{ \coupling\np{\primal,\dual} \LowPlus  
        \bp{ - \kernel\np{\primal,\primalbis}} } \) 
by~\eqref{eq:marginal_kernel} }
      &\leq
        \inf_{ \dualbis \in \DUALBIS} \bgp{
        \kernel_{\dual}^{\couplingbis}\np{\dualbis}
        \UppPlus \SFM{\fonctionprimalbis}{-\couplingbis}\np{\dualbis} } 
        \label{ineq:rfd-post} 
      \intertext{by Fenchel 
inequality~\eqref{eq:Fenchel_inequality_with_general_coupling_un}
where \( \fonctiontrois\np{\primalbis} =
\kernel_{\dual} \np{\primalbis} \) 
}
      & = 
        \inf_{ \dualbis \in \DUALBIS} \bgp{ \sup_{\primalbis \in \PRIMALBIS} 
        \bgp{\couplingbis\np{\primalbis,\dualbis} \LowPlus  
        \bp{ - \kernel_{\dual} \np{\primalbis} } }
        \UppPlus \SFM{\fonctionprimalbis}{-\couplingbis}\np{\dualbis}
}\nonumber
      \intertext{ by definition~\eqref{eq:upper-Fenchel-Moreau_conjugate} 
      of the Fenchel-Moreau $\couplingbis$-conjugate of \( \kernel_{\dual} \)}
      & = 
        \inf_{ \dualbis \in \DUALBIS} \bgp{ \sup_{\primalbis \in \PRIMALBIS} 
        \bgp{\couplingbis\np{\primalbis,\dualbis} \LowPlus  
        \sup_{\primal \in \PRIMAL} 
        \Bp{  \coupling\np{\primal,\dual} 
        \LowPlus  
        \bp{ - \kernel\np{\primal,  \primalbis }}}}
        \UppPlus \SFM{\fonctionprimalbis}{-\couplingbis}\np{\dualbis}
}\nonumber
      \intertext{as \( -\kernel_{\dual} \np{\primalbis} =
\sup_{\primal \in \PRIMAL} \Bp{ \coupling\np{\primal,\dual} \LowPlus  
        \bp{ - \kernel\np{\primal,\primalbis}} } \) 
by~\eqref{eq:marginal_kernel} }
      & = 
\inf_{ \dualbis \in \DUALBIS} \bgp{
        \sup_{\primal \in \PRIMAL, \primalbis \in \PRIMALBIS}
        \bgp{\couplingbis\np{\primalbis,\dualbis} \LowPlus  
        \coupling\np{\primal,\dual} 
        \LowPlus  
        \bp{ - \kernel\np{\primal,  \primalbis }}}\UppPlus \SFM{\fonctionprimalbis}{-\couplingbis}\np{\dualbis}}
\nonumber 
      \intertext{ by the property~\eqref{eq:lower_addition_sup} that the 
operator~$\sup$ is ``distributive'' with respect to~$\LowPlus$,
      and by associativity of~$\LowPlus$}
      &= 
        \inf_{ \dualbis \in \DUALBIS} 
        \bgp{ \SFM{\kernel}{\SumCoupling{\coupling}{\couplingbis}}
\np{\dual,\dualbis}
        \UppPlus \SFM{\fonctionprimalbis}{-\couplingbis}\np{\dualbis}} 
\nonumber 
    \end{align}
  \end{subequations}
by the definition~\eqref{eq:product_conjugate}
of~$\SFM{\kernel}{\SumCoupling{\coupling}{\couplingbis}}$.
  This ends the new proof of Theorem~\ref{th:main}. 
\bigskip

Third, to end the proof of Corollary~\ref{cor:main==extended_couplings}, 
we just check two points.
That inequality~\eqref{ineq:un}
is an equality, by the left hand side assumption in~\eqref{eq:main==}.
That inequality~\eqref{ineq:rfd-post} is also an equality
by assumption~\eqref{eq:rfd-post}.
\end{proof}

\subsubsection{With one bilinear coupling and extended kernel}

We consider the case where one of the two couplings is bilinear, 
whereas the other coupling, the kernel and 
all the other functions can take extended values. 

Let $\PRIMALBIS$ be a locally convex Hausdorff topological vector space
over the real numbers~$\RR$, with its topological dual~$\DUALBIS$
made of continuous linear forms on~$\PRIMALBIS$.
% so that the coupling \( \couplingbis : \PRIMALBIS \times \DUALBIS \to \RR \) is 
The coupling is the duality bilinear form \( \proscal{}{} \), and 
the conjugacy operator on functions is denoted by$\;^\star$.
Let be given $\PRIMAL$ and $\DUAL$, two sets.

\begin{corollary}
  Consider any bivariate function 
  \( \kernel : \PRIMAL \times \PRIMALBIS \to \barRR \)
  and univariate functions 
  \( \fonctionprimal : \PRIMAL  \to \barRR \) and %\( \fonctionprimalbis : \PRIMALBIS  \to \barRR \), 
  \( \fonctionprimalbis : \PRIMALBIS  \to ]-\infty, +\infty ] \), 
  all defined on the ``primal'' sets. 
Let \( \coupling : \PRIMAL \times \DUAL \to \barRR \) be a {coupling} function.
  Suppose that
  \begin{enumerate}
  \item 
the coupling \( \couplingbis : \PRIMALBIS \times \DUALBIS \to \RR \) is 
  the duality bilinear form \( \proscal{}{} \) between 
  $\PRIMALBIS$ and its topological dual~$\DUALBIS$,
  \item 
    the function~$\fonctionprimalbis$ is a proper\footnote{%
      The function~$\fonctionprimalbis$ never takes the value $-\infty$ 
and is not identicaly equal to~$+\infty$.} convex function,
    \label{it:g_proper_convex_function}
  \item 
    for any $\dual\in \DUAL$, 
    the function~$\kernel_{\dual}$ in~\eqref{eq:marginal_kernel}
    is a proper convex function,
    \label{it:marginal_kernel_proper_convex_function}
  \item 
    for any $\dual\in \DUAL$, the function~$\fonctionprimalbis$ is continuous at some point 
    where $\kernel_{\dual}$ is finite.
    \label{it:g_proper_continuous_function}
  \end{enumerate}
  Then, the equality case~\eqref{eq:main==} holds true.
\label{cor:main==bilinear_coupling}
\end{corollary}

\begin{proof}
  The equality case~\eqref{eq:main==} follows by
  checking that the inequalities~\eqref{ineq:un} and~\eqref{ineq:rfd-post} 
  % contained in the proof 
  turn out to be equalities, under the assumptions of
  Corollary~\ref{cor:main==bilinear_coupling}.
  Indeed, the left hand equality
  in~\eqref{eq:main==} gives an equality in~\eqref{ineq:un}.
  The equality in~\eqref{ineq:rfd-post} is a consequence of 
  the equality
  \begin{equation*}
    \inf_{ \primalbis \in \PRIMALBIS} 
    \Bp{ \kernel_{\dual} \np{\primalbis} \UppPlus 
      \fonctionprimalbis\np{\primalbis}} 
    = - \inf_{ \dualbis \in \DUALBIS} 
    \Bp{\kernel_{\dual}^{\couplingbis} \np{\dualbis}
      \UppPlus \SFM{\fonctionprimalbis}{-\couplingbis}\np{\dualbis}}\eqfinv \\
  \end{equation*}
  which holds true by~\cite[Theorem 1]{MR0187062}, 
  under the assumptions made on the  functions~$\fonctionprimalbis$ and $\kernel_{\dual}$,
  where the coupling \( \couplingbis : \PRIMALBIS \times \DUALBIS \to \RR \) is 
  the duality bilinear form \( \proscal{}{} \) between 
  $\PRIMALBIS$ and its dual~$\DUALBIS$.
\end{proof}

\begin{remark} 
  We can weaken the assumptions in Corollary~\ref{cor:main==bilinear_coupling} in two ways.
  \begin{itemize}
  \item 
    Some of the assumptions in 
    item~\ref{it:marginal_kernel_proper_convex_function} 
    in Corollary~\ref{cor:main==bilinear_coupling} --- that bear 
    on the marginal function~$\kernel_{\dual}$ in~\eqref{eq:marginal_kernel} ---
    can be obtained from assumptions
    on the basic elements~$\kernel$ (kernel) and~$\coupling$ (coupling).
    Indeed, if both
    the functions~\( \kernel : \PRIMAL \times \PRIMALBIS \to \barRR \)
    and~\( \coupling(\cdot,\dual) : \PRIMAL \to \barRR \)
    are convex, for any $\dual\in \DUAL$, then 
    the function~$\kernel_{\dual}$ in~\eqref{eq:marginal_kernel} is convex.
  \item 
    The assumptions in 
    item~\ref{it:g_proper_continuous_function}
    in Corollary~\ref{cor:main==bilinear_coupling} 
    can be replaced by the following 
    assumptions (using~\cite[Proposition 15.13]{MR3616647}): 
    $\fonctionprimalbis$ is a proper l.s.c. (lower semi continuous) convex function,
    the marginal function~$\kernel_{\dual}$ in~\eqref{eq:marginal_kernel} 
    is a proper l.s.c. convex function,
    and $0 \in \sri \bp{ \dom \np{\overline\kernel_{\dual}} - \dom \np{\fonctionprimalbis}}$,
    for any $\dual\in \DUAL$.
    Moreover, for any given $\dual\in \DUAL$, 
the marginal function~$\kernel_{\dual}$ in~\eqref{eq:marginal_kernel} 
    is a l.s.c. convex function under the following assumptions
(see~\cite{MR2458436}):
the function~$\kernel$ is l.s.c. convex:
the function $\coupling(\cdot,\dual)$ is proper u.s.c. (upper semi continuous), 
    and the minimization in the definition~\eqref{eq:marginal_kernel} 
of $\kernel_{\dual}$ is performed on a compact set.
  \end{itemize}
\end{remark}

\section{Applications}
\label{Applications}

We now display three applications of our main result 
in Theorem~\ref{th:main}. 
We provide a new formula for the Fenchel-Moreau conjugate of a
generalized inf-convolution.
We obtain formulas with partial Fenchel-Moreau conjugates.
Finally, we consider the Bellman equation in stochastic dynamic 
programming and we  provide a ``Bellman-like'' 
equation for the Fenchel conjugates of the value functions.

\subsection{Fenchel-Moreau  conjugate of generalized inf-convolution}

We generalize the inf-convolution, 
and provide an inequality and an equality with Fenchel-Moreau conjugates
involving three coupling functions.

\begin{definition}
  Let be given three sets $\PRIMAL$, $\PRIMALBIS_1$ and $\PRIMALBIS_2$.
  For any trivariate \emph{convoluting function}
  \begin{equation}
    \convolution : \PRIMALBIS_1 \times \PRIMAL \times \PRIMALBIS_2 \to \barRR 
    % \kernel : \PRIMAL \times \PRIMALBIS_1 \times \PRIMALBIS_2 \to \barRR 
    \eqfinv
    \label{eq:convoluting_inf-convolution}
  \end{equation}
  we define the $\convolution$-\emph{inf-convolution} of two functions
  \( \fonctionprimalbis_1 : \PRIMALBIS_1  \to \barRR \) and 
  \( \fonctionprimalbis_2 : \PRIMALBIS_2  \to \barRR \) by
  \begin{equation}
    \bp{\fonctionprimalbis_1 \overset{\convolution}{\square} 
\fonctionprimalbis_2}
    \np{\primal} =
    \inf_{\primalbis_1 \in \PRIMALBIS_1 , \primalbis_2 \in \PRIMALBIS_2 }
    \bgp{ \fonctionprimalbis_1\np{\primalbis_1} \UppPlus 
      \convolution\np{\primalbis_1,\primal,\primalbis_2} \UppPlus
      \fonctionprimalbis_2\np{\primalbis_2} } 
    \eqsepv \forall \primal \in \PRIMAL 
    \eqfinp
    \label{eq:inf-convolution}
  \end{equation}
\end{definition}

To any convoluting function~\( \convolution \) 
in~\eqref{eq:convoluting_inf-convolution}, 
we can easily attach 
\begin{enumerate}
\item 
  a coupling function~\( \tildeconvolution : \PRIMAL \times 
  \np{ \PRIMALBIS_1 \times \PRIMALBIS_2 }  \to \barRR \) 
  between \( \PRIMAL \) and %\( \PRIMALBIS = 
  \( \PRIMALBIS_1 \times \PRIMALBIS_2 \)
  defined by
  \begin{equation}
    % \PRIMAL
    % \overset{\tilde\convolution}{\leftrightarrow} 
    % \PRIMALBIS_1 \times \PRIMALBIS_2 
    % \mtext{ with }
    \tildeconvolution\bp{\primal,\np{\primalbis_1,\primalbis_2}}=
    \convolution\np{\primalbis_1,\primal,\primalbis_2} 
    \eqsepv \forall \np{\primal,\primalbis_1,\primalbis_2} \in 
    \PRIMAL \times \PRIMALBIS_1 \times \PRIMALBIS_2 
    \eqfinv
\label{eq:tildeconvolution}
  \end{equation}
  % where we have used the notation~\eqref{eq:coupling} for a coupling,
\item 
  a kernel function \( \tildetildeconvolution 
  : \PRIMAL \times \PRIMALBIS_1 \times \PRIMALBIS_2 \to \barRR \)
  defined by 
  \begin{equation}
    \tildetildeconvolution\np{\primal,\primalbis_1,\primalbis_2}=
    \convolution\np{\primalbis_1,\primal,\primalbis_2} 
    \eqsepv \forall \np{\primal,\primalbis_1,\primalbis_2} \in 
    \PRIMAL \times \PRIMALBIS_1 \times \PRIMALBIS_2 
    \eqfinp
\label{eq:tildetildeconvolution}
  \end{equation}
\end{enumerate}
% In what follows, we will use the same notation~\( \convolution \) 
% for the three objects, as the proper meaning can be deduced from the context.

We provide an inequality with Fenchel-Moreau conjugates
involving three coupling functions.

\begin{proposition}
  Let be given three ``primal'' sets $\PRIMAL$, $\PRIMALBIS_1$, $\PRIMALBIS_2$ 
and three ``dual'' sets  $\DUAL$, $\DUALBIS_1$, $\DUALBIS_2$, 
  together with three coupling functions
  \begin{equation}
    \coupling : \PRIMAL \times \DUAL \to \barRR \eqsepv
    \couplingbis_1 : \PRIMALBIS_1 \times \DUALBIS_1 \to \barRR \eqsepv
    \couplingbis_2 : \PRIMALBIS_2 \times \DUALBIS_2 \to \barRR 
    \eqfinp 
  \end{equation}
  For any univariate functions 
  \( \fonctionprimal : \PRIMAL  \to \barRR \),
  \( \fonctionprimalbis_1 : \PRIMALBIS_1  \to \barRR \) and 
  \( \fonctionprimalbis_2 : \PRIMALBIS_2  \to \barRR \),
  all defined on the ``primal'' sets, 
  we have that
  \begin{equation}
    \fonctionprimal\np{\primal} \geq 
    \bp{\fonctionprimalbis_1  \overset{\convolution}{\square} \fonctionprimalbis_2}
    \np{\primal}     \eqsepv \forall \primal \in \PRIMAL 
    \Rightarrow 
    \SFM{\fonctionprimal}{\coupling}\np{\dual} \leq 
    \bp{ \SFM{\fonctionprimalbis_1}{(-\couplingbis_1)}  
      \overset{ {\convolution}^{\sharp} }{\square} 
      \SFM{\fonctionprimalbis_2}{(-\couplingbis_2)} }\np{\dual}
    \eqsepv \forall \dual \in \DUAL 
    \eqfinv
    \label{eq:main_inf-convolution}
  \end{equation}
  where the convoluting function~\( {\convolution}^{\sharp} \)
  on the ``dual'' sets is given by
  \begin{subequations}
  \begin{equation}
    {\convolution}^{\sharp} =  \SFM{\tildetildeconvolution}%
    {\SumCoupling{\coupling}{\couplingbis_1\LowPlus\couplingbis_2} } 
    \eqfinv
  \end{equation}
(where the kernel~\( \tildetildeconvolution \) 
is defined in~\eqref{eq:tildetildeconvolution}),
  that is, by
  \begin{multline}
    {\convolution}^{\sharp}\np{\dualbis_1,\dual,\dualbis_2}
    = \\ 
    \sup_{\np{\primalbis_1,\primal,\primalbis_2} \in 
      \PRIMALBIS_1 \times \PRIMAL \times \PRIMALBIS_2 }
    \Bp{ 
      \coupling\np{\primal,\dual} 
      \LowPlus 
      \couplingbis_1\np{\primalbis_1,\dualbis_1} \LowPlus
      \couplingbis_2\np{\primalbis_2,\dualbis_2} 
      % \Bp{ -\bp{ \couplingbis_1\np{\primalbis_1,\dualbis_1} \LowPlus
      % \couplingbis_2\np{\primalbis_2,\dualbis_2} } }
      \LowPlus 
      \bp{ - \convolution\np{\primalbis_1,\primal,\primalbis_2} } } 
    \eqfinp
    \label{eq:dual-convolution}
  \end{multline}
  \end{subequations}
\end{proposition}

\begin{proof}
  The left hand side assumption in~\eqref{eq:main_inf-convolution}
  can be rewritten as 
  \begin{equation}
    \fonctionprimal\np{\primal} \geq 
    \inf_{\primalbis_1 \in \PRIMALBIS_1 , \primalbis_2 \in \PRIMALBIS_2 }
    \bgp{ \tildetildeconvolution\bp{\primal,\primalbis_1,\primalbis_2} \UppPlus
      \Bp{ \fonctionprimalbis_1\np{\primalbis_1} \UppPlus 
        \fonctionprimalbis_2\np{\primalbis_2} } } 
    \eqsepv \forall \primal \in \PRIMAL 
\eqfinp
\label{eq:left_hand_side_assumption_in_main_inf-convolution}
  \end{equation}
  Now, we apply Theorem~\ref{th:main} with 
  \begin{align*}
    \PRIMALBIS &= \PRIMALBIS_1 \times \PRIMALBIS_2 \eqsepv
 \DUALBIS = \DUALBIS_1 \times \DUALBIS_2 
\\
\couplingbis\bp{\np{\primalbis_1,\primalbis_2},\np{\dualbis_1,\dualbis_2}}
  &=
\couplingbis_1\np{\primalbis_1,\dualbis_1} \LowPlus
  \couplingbis_2\np{\primalbis_2,\dualbis_2} 
\\
\fonctionprimalbis\np{\primalbis_1,\primalbis_2}
&=
\fonctionprimalbis_1\np{\primalbis_1} \UppPlus 
  \fonctionprimalbis_2\np{\primalbis_2} 
\mtext{ and } \kernel = 
\tildetildeconvolution \mtext{ by~\eqref{eq:tildetildeconvolution}. }
%\eqfinp 
  \end{align*}
  We first prove that 
  \begin{equation}
    \SFM{\fonctionprimalbis}{-\np{\couplingbis_1\LowPlus\couplingbis_2}}
    \np{\dualbis_1,\dualbis_2} \leq
    \Bp{ \SFM{\fonctionprimalbis_1}{(-\couplingbis_1)}\np{\dualbis_1}
      \UppPlus 
      \SFM{\fonctionprimalbis_2}{(-\couplingbis_2)}\np{\dualbis_2}} 
    \eqfinp
    \label{eq:psi-additive}
  \end{equation}
For this, we let the reader check that 
the following preliminary inequality always holds true
\begin{equation}
  \bp{ -\np{u_1 \LowPlus u_2} } \LowPlus \bp{ -\np{v_1 \UppPlus v_2} }
\leq
  \bp{ \np{-u_1} \LowPlus \np{-v_1} } \UppPlus 
\bp{ \np{-u_2} \LowPlus \np{-v_2} } 
\eqfinp
    \label{eq:psi-additive_preliminary_inequality}
\end{equation}
  Then, we have that 
  \begin{align*}
    \SFM{\fonctionprimalbis}{-\np{\couplingbis_1\LowPlus\couplingbis_2}}
    \np{\dualbis_1,\dualbis_2} 
    &= 
      \sup_{\primalbis_1\in \PRIMALBIS_1,\primalbis_2\in \PRIMALBIS_2}
    \bgp{
\Bp{ -\bp{\couplingbis_1\np{\primalbis_1,\dualbis_1}
      \LowPlus\couplingbis_2\np{\primalbis_2,\dualbis_2} } }
      \LowPlus 
 \Bp{ - \bp{ \fonctionprimalbis_1\np{\primalbis_1} \UppPlus 
      \fonctionprimalbis_2\np{\primalbis_2} } } }
      \nonumber 
    % \\
    % &= 
    %   \sup_{\primalbis_1\in \PRIMALBIS_1,\primalbis_2\in \PRIMALBIS_2}
    %   \Bgp{
    %   \Bp{
    %   \bp{ {- \couplingbis_1\np{\primalbis_1,\dualbis_1}} \UppPlus {- \couplingbis_2\np{\primalbis_2,\dualbis_2}}}
    %   \LowPlus \np{- \fonctionprimalbis_2\np{\primalbis_2}}} \LowPlus 
    %   \np{ - \fonctionprimalbis_1\np{\primalbis_1}}} 
    %   \nonumber 
    %   \intertext{by~\eqref{eq:lower_upper_addition_minus} 
    %   and associativity of~$\LowPlus$}
    % % 
    % &\leq 
    %   \sup_{\primalbis_1\in \PRIMALBIS_1,\primalbis_2\in \PRIMALBIS_2}
    %   \Bgp{
    %   \Bp{ - \couplingbis_1\np{\primalbis_1,\dualbis_1} \UppPlus 
    %   \bp{ - \couplingbis_2\np{\primalbis_2,\dualbis_2} \LowPlus 
    %   \np{- \fonctionprimalbis_2\np{\primalbis_2}}}}
    %   \LowPlus 
    %   \np{ - \fonctionprimalbis_1\np{\primalbis_1}}}
    %   \tag{by~\eqref{eq:lower_upper_addition_inequality}} 
    %   \nonumber 
    % \\
    % &\leq 
    %   \sup_{\primalbis_1\in \PRIMALBIS_1,\primalbis_2\in \PRIMALBIS_2}
    %   \Bp{
    %   \bp{ - \couplingbis_1\np{\primalbis_1,\dualbis_1} \LowPlus 
    %   \np{-\fonctionprimalbis_1\np{\primalbis_1}}}
    %   \UppPlus \bp{ - \couplingbis_2\np{\primalbis_2,\dualbis_2} 
    %   \LowPlus \np{- \fonctionprimalbis_2\np{\primalbis_2}}}}
    %   \tag{by~\eqref{eq:lower_upper_addition_inequality}}
    %   \nonumber
    \\
    &\leq
      \sup_{\primalbis_1\in \PRIMALBIS_1} 
      \Bp{
      \bp{ -\couplingbis_1\np{\primalbis_1,\dualbis_1} }
      \LowPlus \bp{ -\fonctionprimalbis_1\np{\primalbis_1} } }
      \UppPlus 
      \sup_{\primalbis_2\in \PRIMALBIS_2} 
      \Bp{
      \bp{ -\couplingbis_2\np{\primalbis_2,\dualbis_2} 
      \LowPlus \bp{ -\fonctionprimalbis_2\np{\primalbis_2} } } }
\intertext{ by the preliminary 
inequality~\eqref{eq:psi-additive_preliminary_inequality}}
      \nonumber 
    &=  
      \SFM{\fonctionprimalbis_1}{(-\couplingbis_1)}\np{\dualbis_1}
      \UppPlus 
      \SFM{\fonctionprimalbis_2}{(-\couplingbis_2)}\np{\dualbis_2} 
      \eqfinp
      \nonumber 
  \end{align*}
  We now obtain, by~\eqref{eq:main} 
applied with~\eqref{eq:left_hand_side_assumption_in_main_inf-convolution},
  \begin{subequations}
    \begin{align}
      \SFM{\fonctionprimal}{\coupling}\np{\dual} 
      &\leq 
        \inf_{\dualbis_1 \in \DUALBIS_1 , \dualbis_2 \in \DUALBIS_2 }
        \bgp{ \SFM{\tildetildeconvolution}%
        {\SumCoupling{\coupling}{\np{\couplingbis_1\LowPlus\couplingbis_2}} }
        \np{\dual,\dualbis_1,\dualbis_2} \UppPlus
        \SFM{\fonctionprimalbis}{-\np{\couplingbis_1\LowPlus\couplingbis_2}}
        \np{\dualbis_1,\dualbis_2} } 
        \nonumber 
        \nonumber \\
      &\leq
        \inf_{\dualbis_1 \in \DUALBIS_1 , \dualbis_2 \in \DUALBIS_2 }
        \bgp{ \SFM{\tildetildeconvolution}%
        {\SumCoupling{\coupling}{\np{\couplingbis_1\LowPlus\couplingbis_2}} }
        \np{\dual,\dualbis_1,\dualbis_2} \UppPlus
        \Bp{ \SFM{\fonctionprimalbis_1}{(-\couplingbis_1)}\np{\dualbis_1}
        \UppPlus 
        \SFM{\fonctionprimalbis_2}{(-\couplingbis_2)}\np{\dualbis_2} } } 
        \tag{by~\eqref{eq:psi-additive}}
        \nonumber 
      \\
      &=
        \bp{ \SFM{\fonctionprimalbis_1}{(-\couplingbis_1)}  
        \overset{ {\convolution}^{\sharp} }{\square} 
        \SFM{\fonctionprimalbis_2}{(-\couplingbis_2)} }\np{\dual} 
        \nonumber
    \end{align}
  \end{subequations}
by definition~\eqref{eq:inf-convolution} of the generalized inf-convolution,
and where the kernel~\( \tildetildeconvolution \) 
is defined in~\eqref{eq:tildetildeconvolution}.

This ends the proof.
\end{proof}

We check our result in Theorem~\ref{th:main} on
the classical inf-convolution.
Suppose that $\PRIMAL=\PRIMALBIS = \RR^n$
and $\DUAL=\DUALBIS = \RR^n$ equipped with the scalar product \( \proscal{}{} \). 
The conjugacy operator on functions is denoted by$\;^\star$.
For the \emph{characteristic function} of a subset
\( \Uncertain \subset \UNCERTAIN \) of a set~$\UNCERTAIN$, we adopt the notation
\begin{equation}
  \delta_{\Uncertain} : \UNCERTAIN \to \{ 0, +\infty \} \eqsepv
  \delta_{\Uncertain}\np{\uncertain} =  
   \begin{cases}
      0 &\text{ if } \uncertain \in \Uncertain \eqfinv \\ 
      +\infty &\text{ if } \uncertain \not\in \Uncertain \eqfinp
    \end{cases} 
\end{equation}
Now, when we take 
\begin{subequations}
  \begin{equation}
    \convolution\np{\primalbis_1,\primal,\primalbis_2}=
    \delta_{\primalbis_1 \UppPlus \primalbis_2}\np{\primal}
    \eqsepv
    \coupling\np{\primal,\dual} = \proscal{\primal}{\dual} 
    \eqsepv
    \couplingbis_i\np{\primalbis_i,\dualbis_i} = -\proscal{\primalbis_i}{\dualbis_i} 
    \eqsepv i=1,2 
    \eqfinv
  \end{equation}
  we find that, by~\eqref{eq:dual-convolution}
  % \begin{equation}
  %   \SFM{\convolution\np{\primalbis_1,\cdot,\primalbis_2}}{\coupling}\np{\dual}
  %   = \proscal{\primalbis_1}{\dual} \UppPlus \proscal{\primalbis_2}{\dual} 
  %   \eqfinv
  % \end{equation}
  % and also that 
  \begin{equation}
    {\convolution}^{\sharp}\np{\dualbis_1,\dual,\dualbis_2}
    = \delta_{\dualbis_1}\np{\dual} \LowPlus \delta_{\dualbis_1}\np{\dual} 
    \eqsepv
    \bp{ \SFM{\fonctionprimalbis_1}{(-\couplingbis_1)}  
      \overset{ {\convolution}^{\sharp} }{\square} 
      \SFM{\fonctionprimalbis_2}{(-\couplingbis_2)} }\np{\dual} 
    = \fonctionprimalbis_1^\star\np{\dual} 
    \UppPlus \fonctionprimalbis_2^\star\np{\dual}
    \eqfinp
  \end{equation}
\end{subequations}
We conclude with~\eqref{eq:main_inf-convolution} 
that we indeed obtain the well known property of the inf-convolution:
\( \fonctionprimal \ge \fonctionprimalbis_1 \square \fonctionprimalbis_2 
\Rightarrow %implies that 
\fonctionprimal^\star \le 
(\fonctionprimalbis_1 \square \fonctionprimalbis_2)^\star 
= \fonctionprimalbis_1^\star + \fonctionprimalbis_2^\star \).
\bigskip

To end up, we provide an expression of the inf-convolution as a 
Fenchel-Moreau conjugate, and we obtain an equality 
with Fenchel-Moreau conjugates
involving three coupling functions.

\begin{proposition}
  The $\convolution$-\emph{inf-convolution} 
in~\eqref{eq:inf-convolution} is given by
  \begin{equation}
    \fonctionprimalbis_1  \overset{\convolution}{\square} \fonctionprimalbis_2 
    = - \SFM{\np{\fonctionprimalbis_1 \UppPlus \fonctionprimalbis_2}}%
    {-\tildeconvolution} 
    \eqfinp
    \label{eq:kernel_inf-convolution}
  \end{equation}     
\end{proposition}

\begin{proof}
  \begin{subequations}
    For any \( \primal \in \PRIMAL \), we have that 
    \begin{align*}
      \bp{\fonctionprimalbis_1 \overset{\convolution}{\square} 
\fonctionprimalbis_2}  \np{\primal} 
      &=
        \inf_{\primalbis_1 \in \PRIMALBIS_1 , \primalbis_2 \in \PRIMALBIS_2 }
        \bgp{ \fonctionprimalbis_1\np{\primalbis_1} \UppPlus 
        \convolution\np{\primalbis_1,\primal,\primalbis_2} \UppPlus
        \fonctionprimalbis_2\np{\primalbis_2} }
\intertext{ by definition~\eqref{eq:inf-convolution} 
of the generalized inf-convolution }
      &=
        \inf_{\primalbis_1 \in \PRIMALBIS_1 , \primalbis_2 \in \PRIMALBIS_2 }
        \Bp{ \tildeconvolution\bp{\primal,\np{\primalbis_1,\primalbis_2}} 
\UppPlus
        \bp{ \fonctionprimalbis_1\np{\primalbis_1} \UppPlus 
        \fonctionprimalbis_2\np{\primalbis_2} } } 
\intertext{ by definition~\eqref{eq:tildeconvolution} 
of the coupling function $\tildeconvolution$ }
      &=
        - \sup_{\primalbis_1 \in \PRIMALBIS_1 , \primalbis_2 \in \PRIMALBIS_2 }
        \bgp{ \Bp{ -\tildeconvolution\bp{\primal,\np{\primalbis_1,\primalbis_2}} } 
        \LowPlus
        \Bp{-\bp{ \fonctionprimalbis_1\np{\primalbis_1} \UppPlus 
        \fonctionprimalbis_2\np{\primalbis_2} } } } \\
      &=
        - \bp{\SFM{\np{\fonctionprimalbis_1 \UppPlus \fonctionprimalbis_2}}%
        {-\tildeconvolution}}
        \np{\primal}  \eqfinp
    \end{align*}
  \end{subequations}
This ends the proof.
\end{proof}

\begin{proposition}
  If there exist two {coupling} functions
  \begin{equation}
    \couplingter_1 : \DUAL \times \PRIMALBIS_1 \to \barRR \eqsepv
    \couplingter_2 : \DUAL \times \PRIMALBIS_2 \to \barRR \eqfinv
  \end{equation}
  such that the $\coupling$-Fenchel-Moreau conjugate of 
the convoluting function~$\convolution$ splits as 
  \begin{equation}
    % \SFM{\kernel\np{\cdot,\primalbis_1,\primalbis_2}}{\coupling}\np{\dual}
    \SFM{\convolution\np{\primalbis_1,\cdot,\primalbis_2}}{\coupling}\np{\dual}
    = \couplingter_1\np{\dual,\primalbis_1} \LowPlus
    \couplingter_2\np{\dual,\primalbis_2} \eqsepv
    \forall \np{\dual,\primalbis_1,\primalbis_2} \in 
    \DUAL \times \PRIMALBIS_1 \times \PRIMALBIS_2 
    \eqfinv
\label{eq:splits}
  \end{equation}
  then the $\coupling$-Fenchel-Moreau conjugate of 
the inf-convolution 
\( \fonctionprimalbis_1 \overset{\convolution}{\square} 
        \fonctionprimalbis_2 \) is given by a sum as
  \begin{equation}
    \SFM{\bp{\fonctionprimalbis_1 \overset{\convolution}{\square} 
        \fonctionprimalbis_2}}{\coupling} =
    \SFM{\fonctionprimalbis_1}{\couplingter_1} \LowPlus
    \SFM{\fonctionprimalbis_2}{\couplingter_2} 
    \eqfinp
  \end{equation}
\end{proposition}

\begin{proof}
We have that 
  \begin{subequations}
    \begin{align*}
      \SFM{\bp{\fonctionprimalbis_1 \overset{\convolution}{\square} 
      \fonctionprimalbis_2}}{\coupling}\np{\dual}
      &=  
        \SFM{ \bp{ - \SFM{\np{\fonctionprimalbis_1 \UppPlus 
\fonctionprimalbis_2}}%
        {-\tildeconvolution} } }{\coupling} \np{\dual}
\tag{ by~\eqref{eq:kernel_inf-convolution} }
      \\
      &=  
        \SFM{\np{\fonctionprimalbis_1 \UppPlus \fonctionprimalbis_2} }%
        { \np{-\tildeconvolution} \LowPlus \coupling} \np{\dual}
\tag{ by~\eqref{eq:composition_of_conjugates}} 
      \\
      &=  
        \sup_{\np{\primalbis_1,\primal,\primalbis_2} \in 
        \PRIMALBIS_1 \times \PRIMAL \times \PRIMALBIS_2 }
        \bgp{ 
        \bp{-\convolution\np{\primalbis_1,\primal,\primalbis_2} } 
        \LowPlus \coupling\np{\primal,\dual} 
        \LowPlus \Bp{-\bp{\fonctionprimalbis_1\np{\primalbis_1} 
        \UppPlus \fonctionprimalbis_2\np{\primalbis_2} } } }
      \\
      &=  
        \sup_{\np{\primalbis_1,\primalbis_2} \in 
        \PRIMALBIS_1 \times \PRIMALBIS_2 }
        \Bgp{ \sup_{\primal \in \PRIMAL}
        \bgp{ 
        \bp{-\convolution\np{\primalbis_1,\primal,\primalbis_2} } 
        \LowPlus \coupling\np{\primal,\dual} }
        \LowPlus \Bp{-\bp{\fonctionprimalbis_1\np{\primalbis_1} 
        \UppPlus \fonctionprimalbis_2\np{\primalbis_2} } } } 
      \\
      &=  
        \sup_{\np{\primalbis_1,\primalbis_2} \in 
        \PRIMALBIS_1 \times \PRIMALBIS_2 }
        \Bgp{ 
        \SFM{\convolution\np{\primalbis_1,\cdot,\primalbis_2}}{\coupling}\np{\dual}
        \LowPlus \Bp{-\bp{\fonctionprimalbis_1\np{\primalbis_1} 
        \UppPlus \fonctionprimalbis_2\np{\primalbis_2} } } } 
      \\
      &=  
        \sup_{\np{\primalbis_1,\primalbis_2} \in 
        \PRIMALBIS_1 \times \PRIMALBIS_2 }
        \Bgp{ 
        \Bp{ \couplingter_1\np{\dual,\primalbis_1} \LowPlus
        \couplingter_2\np{\dual,\primalbis_2} }
        \LowPlus \Bp{-\bp{\fonctionprimalbis_1\np{\primalbis_1} 
        \UppPlus \fonctionprimalbis_2\np{\primalbis_2} } } } 
\tag{ by assumption~\eqref{eq:splits} }
      \\
      &=  
        \sup_{\np{\primalbis_1,\primalbis_2} \in 
        \PRIMALBIS_1 \times \PRIMALBIS_2 }
        \Bgp{ 
        \couplingter_1\np{\dual,\primalbis_1} \LowPlus
        \couplingter_2\np{\dual,\primalbis_2} 
        \LowPlus \bp{-\fonctionprimalbis_1\np{\primalbis_1}}
        \LowPlus \bp{-\fonctionprimalbis_2\np{\primalbis_2} } }
      \\
      &=  
        \SFM{\fonctionprimalbis_1}{\couplingter_1}\np{\dual} \LowPlus
        \SFM{\fonctionprimalbis_2}{\couplingter_2}\np{\dual}
        \eqfinp
    \end{align*}
  \end{subequations}
This ends the proof.
\end{proof}

\subsection{Exchanging partial Fenchel-Moreau conjugates}

Let be given two ``primal'' sets $\PRIMAL$, $\PRIMALBIS$
and two ``dual'' sets $\DUAL$, $\DUALBIS$, 
together with two {coupling} functions
\begin{equation}
  \coupling : \PRIMAL \times \DUAL \to \barRR \eqsepv
  \couplingbis : \PRIMALBIS \times \DUALBIS \to \barRR \eqfinp 
\end{equation}

In~\eqref{eq:main}, all Fenchel-Moreau conjugates stand on the right side
of the implication. We show formulas where they appear on both sides.
\begin{subequations}
  For this purpose, for any \emph{exchange function}
  \begin{equation}
    \Exchange : \PRIMAL \times \DUALBIS \to \barRR \eqfinv
    \label{eq:ExchangePRIMALtimesDUALBIS}
  \end{equation}  
  we introduce the \emph{partial Fenchel-Moreau conjugates}
  \begin{align}
    \SFM{ \bp{ -\Exchange\np{\primal,\cdot} } }{\couplingbis}
    \np{\primalbis} 
    &= \sup_{\dualbis \in \DUALBIS} 
      \Bp{ \couplingbis\np{\primalbis,\dualbis}
      \LowPlus \Exchange\np{\primal,\dualbis} } 
      \eqsepv 
      \forall \np{\primal, \primalbis} \in \PRIMAL \times \PRIMALBIS 
      \eqfinv
      \label{eq:partial_Fenchel-Moreau_conjugate_couplingbis}
    \\
    \SFM{\Exchange\np{\cdot,\dualbis}}{\coupling}\np{\dual} 
    &= \sup_{\primal \in \PRIMAL}
      \Bp{ 
      \coupling\np{\primal,\dual} 
      \LowPlus  \bp{ - \Exchange\np{\primal,\dualbis} } } 
      \eqsepv \forall \np{\dual,\dualbis} \in \DUAL \times \DUALBIS 
      \eqfinp
      \label{eq:partial_Fenchel-Moreau_conjugate_coupling}
  \end{align}
\end{subequations}

We prove the following implication.
\begin{proposition}
  For any function 
  \( \Exchange : \PRIMAL \times \DUALBIS \to \barRR \),
  we have that 
  \begin{multline}
    \fonctionprimal\np{\primal} \geq \inf_{\primalbis \in \PRIMALBIS} \Bp{ 
      \SFM{ \Exchange\np{\primal,\cdot} }{\couplingbis}
      \np{\primalbis} 
      \UppPlus \fonctionprimalbis\np{\primalbis} } 
    \eqsepv \forall \primal \in \PRIMAL 
    \Rightarrow \\ 
    \SFM{\fonctionprimal}{\coupling}\np{\dual} 
    \leq 
    \inf_{\dualbis \in \DUALBIS} 
    \Bp{\SFM{ \bp{-\Exchange\np{\cdot,\dualbis}}}{\coupling}\np{\dual} 
      \UppPlus \SFM{\fonctionprimalbis}{-\couplingbis}\np{\dualbis} } 
    \eqsepv \forall \dual \in \DUAL 
    \eqfinp
  \end{multline}
\end{proposition}

\begin{proof}
  We use the following 
  Lemma~\ref{lem:partial_Fenchel-Moreau_conjugate_coupling_kernel}.
  We apply Theorem~\ref{th:main} with 
  the function
\( \kernel\np{\primal, \primalbis} =
    \SFM{ \bp{-\Exchange\np{\primal,\cdot}} }{\couplingbis}
    \np{\primalbis} \)
 defined by equality in the left hand side
  inequality in~\eqref{eq:partial_Fenchel-Moreau_conjugate_coupling_kernel}.
  Then, we insert the right hand side inequality 
  in~\eqref{eq:partial_Fenchel-Moreau_conjugate_coupling_kernel}
  % \begin{equation}
  %   \SFM{\kernel}{\SumCoupling{\coupling}{\couplingbis}}\np{\dual,\dualbis} \leq
  %   \SFM{\Exchange\np{\cdot,\dualbis}}{\coupling}\np{\dual} \eqsepv 
  %   \forall \np{\dual,\dualbis} \in \DUAL \times \DUALBIS 
  %   \label{hamiltonien-dual}
  % \end{equation}
  into implication~\eqref{eq:main}.
\end{proof}

\begin{lemma}
  For any function 
  \( \Exchange : \PRIMAL \times \DUALBIS \to \barRR \),
  we have that 
  \begin{multline}
    \kernel\np{\primal, \primalbis} \geq
    \SFM{ \bp{-\Exchange\np{\primal,\cdot}} }{\couplingbis}
    \np{\primalbis} \eqsepv  
    \forall \np{\primal, \primalbis} \in \PRIMAL \times \PRIMALBIS 
    \Rightarrow \\
    \SFM{\kernel}{\SumCoupling{\coupling}{\couplingbis}}\np{\dual,\dualbis} \leq
    \SFM{\Exchange\np{\cdot,\dualbis}}{\coupling}\np{\dual} \eqsepv 
    \forall \np{\dual,\dualbis} \in \DUAL \times \DUALBIS 
    \eqfinp
    \label{eq:partial_Fenchel-Moreau_conjugate_coupling_kernel}
  \end{multline}
  \label{lem:partial_Fenchel-Moreau_conjugate_coupling_kernel}
\end{lemma}

\begin{proof}
  Supposing that 
  \begin{equation}
    \kernel\np{\primal, \primalbis} \geq
    \SFM{ \bp{-\Exchange\np{\primal,\cdot} } }{\couplingbis}
    \np{\primalbis} \eqsepv  
    \forall \np{\primal, \primalbis} \in \PRIMAL \times \PRIMALBIS 
    \eqfinv 
    \label{eq:ExchangePRIMALtimesDUALBISkernel}
  \end{equation}
  we calculate, for all \( \np{\dual,\dualbis} \in \DUAL \times \DUALBIS \), 
  \begin{subequations}
    \begin{align*}
      \SFM{\kernel}{\SumCoupling{\coupling}{\couplingbis}}\np{\dual,\dualbis} 
      &= \sup_{\primal \in \PRIMAL, \primalbis \in \PRIMALBIS} 
        \Bp{ 
        \coupling\np{\primal,\dual} \LowPlus 
        \couplingbis\np{\primalbis,\dualbis}
        \LowPlus \bp{ - \kernel \np{\primal, \primalbis}}} 
        \intertext{ by the definition~\eqref{eq:product_conjugate}
        of~$\SFM{\kernel}{\SumCoupling{\coupling}{\couplingbis}}$} 
      & \leq
        \sup_{\primal \in \PRIMAL, \primalbis \in \PRIMALBIS} 
        \Bp{ 
        \coupling\np{\primal,\dual} \LowPlus 
        \couplingbis\np{\primalbis,\dualbis}
        \LowPlus 
        \SFM{ \Exchange\np{\primal,\cdot}
        }{\couplingbis} \np{\primalbis} }
        \intertext{ by inequality~\eqref{eq:ExchangePRIMALtimesDUALBISkernel} 
        for~$\kernel$ }
      & =
        \sup_{\primal \in \PRIMAL} \Bgp{
        \coupling\np{\primal,\dual} \LowPlus 
        \sup_{\primalbis \in \PRIMALBIS} 
        \bgp{ 
        \couplingbis\np{\primalbis,\dualbis}
        \LowPlus  \Bp{ - \bp{ -  
        \SFM{ \Exchange\np{\primal,\cdot}
        }{\couplingbis} \np{\primalbis}} } } }
        \intertext{ by the property~\eqref{eq:lower_addition_sup} that the 
operator~$\sup$ is ``distributive'' with respect to~$\LowPlus$ }
      & =
        \sup_{\primal \in \PRIMAL} \bgp{
        \coupling\np{\primal,\dual} \LowPlus 
        \SFM{ \bp{-\Exchange\np{\primal,\cdot}}
        }{\couplingbis\couplingbis}\np{\dualbis}
        }
        \intertext{ by definition~\eqref{eq:Fenchel-Moreau_biconjugate} of the 
        biconjugate }
      & \leq 
        \sup_{\primal \in \PRIMAL} \bgp{
        \coupling\np{\primal,\dual} \LowPlus 
        \bp{-\Exchange\np{\primal,\dualbis}}
        } 
        \intertext{  by the 
      inequality~\eqref{eq:Fenchel-Moreau_biconjugate_inequality} 
      between a function and its biconjugate, 
        and by the property~\eqref{eq:lower_addition_leq} that 
        the operator~$\LowPlus$ is monotone }
      &= 
        \SFM{ \Exchange\np{\cdot,\dualbis}}{\coupling}\np{\dual} 
    \end{align*}
        %         \intertext{ 
    by the 
    definition~\eqref{eq:partial_Fenchel-Moreau_conjugate_coupling}
    of partial Fenchel-Moreau conjugate. %} %\nonumber
  \end{subequations}
\end{proof}

\subsection{Fenchel conjugates of Bellman functions}

We consider the Bellman equation in stochastic dynamic 
programming and we  provide a ``Bellman-like'' 
equation for the Fenchel conjugates of the value functions.
Related works are \cite{Rockafellar-Wolenski:2001b} and 
\cite{Pennanen-Ari-Pekka:2018}.

\subsubsection{Basic sets and couplings}

Let $\epro$ be a probability space. 
Let $ 1 \leq p<+\infty$ and $q$ be defined by \(1/p+1/q=1 \).
Adopting the notation of 
Sect.~\ref{A_duality_formula_with_multiple_Fenchel-Moreau_conjugates},
%we put $\PRIMAL=\STATE=\RR^{n_{\STATE}}$ and 
we put $\STATE=\RR^{n_{\STATE}}$ and 
$\PRIMALBIS = \espaceL{p}$ 
the space of $p$-integrable
random variables with values in~$\RR^{n_{\STATE}}$.
Elements of $\PRIMALBIS$, that is, 
$p$-integrable random variables with values in~$\STATE$,
will be denoted by bold letters like~$\va{\State}$
and elements of $\DUALBIS = \espaceL{q}$ 
by~$\va{\Dualstate}$.

The coupling~$\coupling$ between 
%$\PRIMAL=\STATE=\RR^{n_{\STATE}}$ and $\DUAL=\DUALSTATE=\RR^{n_{\STATE}}$ 
$\STATE=\RR^{n_{\STATE}}$ and $\DUALSTATE=\RR^{n_{\STATE}}$ 
is the usual scalar product \( \proscal{}{} \).
The coupling~$\couplingbis$ between 
$\PRIMALBIS = \espaceL{p}$ and
$\DUALBIS  = \espaceL{q}$ 
is naturally derived in such a way that
\begin{subequations}
\begin{align}
  \coupling\np{\state,\dualstate}&=\proscal{\state}{\dualstate} \eqsepv
\forall \np{\state,\dualstate} \in \STATE\times\DUALSTATE \eqsepv \\
  \couplingbis\np{\va{\State},\va{\Dualstate}}&=
  \EE\bc{\proscal{\va{\State}}{\va{\Dualstate}}} \eqsepv
\forall \np{\va{\State},\va{\Dualstate}} \in \PRIMALBIS\times\DUALBIS
  \eqfinp 
\end{align}
  \label{eq:coupling:Bellman}
  \end{subequations}
In that case, the conjugates \( \SFM{\fonctionprimal}{\coupling} \),
\( \SFM{\fonctionprimalbis}{\couplingbis} \), 
\( \SFM{\fonctionprimalbis}{-\couplingbis} \)  and
\(  \SFM{ \kernel }{\SumCoupling{\coupling}{(-\couplingbis)}} \)
are denoted by
\( \LFM{\fonctionprimal} \), \( \LFM{\fonctionprimalbis} \),
\( \SFM{\fonctionprimalbis}{(-\star)} \) and
\(  \SFM{ \kernel }{\SumCoupling{\star}{(-\star)}} \).
 One can find such a difference coupling in~\cite{Volle:1983}.

\subsubsection{Bellman functions and Bellman equation}

Let time~$t=0,1,\ldots,\horizon$ be discrete, with $\horizon \in \NN^*$.
Consider a stochastic optimal control problem 
with state space~$\STATE=\RR^{n_{\STATE}}$, 
control space~$\CONTROL=\RR^{n_{\CONTROL}}$,
noise process~$\sequence{ \va{\Uncertain}_{t} }{t=1,\ldots,\horizon}$
taking values in~$\UNCERTAIN=\RR^{n_{\UNCERTAIN}}$ and defined over 
the probability space~\( \epro \).

For each time~$t=0,1,\dots,\horizon-1$, we have a dynamics 
\( \dynamics_{t} : \STATE \times \CONTROL \times \UNCERTAIN \to \STATE \)
and an instantaneous cost
\( \coutint_{t} : \STATE \times \CONTROL \times \UNCERTAIN \to 
]-\infty,+\infty] \);
we also have a final cost \( \coutfin: \STATE \to ]-\infty,+\infty] \).
These two costs can take the value $+\infty$, so that we can easily handle
state and control constraints.  
\begin{assumption}
  We make the following assumptions:
  \begin{enumerate}
  \item 
    for any \( \np{\state, \control} \in \STATE \times \CONTROL \), 
the $\RR^{n_{\STATE}}$-valued 
    random variable 
    \( \dynamics_t(\state,\control,\va{\Uncertain}_{t+1}) \)
    belongs to \( \espaceL{p} \),
    \label{it:assumption:Bellman1}
  \item
    the instantaneous costs
    \( \coutint_{t} : \STATE \times \CONTROL \times \UNCERTAIN \to \bRR \),
    for $t=0,\ldots,\horizon-1$, and
    the final cost \( \coutfin: \STATE \to \bRR \) are 
nonnegative\footnote{%
We could also consider functions that are uniformly bounded below.
However, for the sake of simplicity, and without loss of generality, 
we will deal with nonnegative functions.} measurable functions.
    \label{it:assumption:Bellman2}
  \end{enumerate}
  \label{assumption:Bellman}
\end{assumption}

By item~\ref{it:assumption:Bellman2} 
in Assumption~\ref{assumption:Bellman}, 
we can define \emph{Bellman functions} by,
for all \( \state \in \STATE \), 
\begin{subequations}
  \begin{align}
    \Value_{\horizon}\np{\state} &= \coutfin\np{\state} \eqfinv \\
%    \label{eq:bellman}
    \Value_t\np{\state} &= \inf_{\va{\State},\va{\Control}}
                          \EE\bc{ \sum_{s=t}^{\horizon-1} 
                          \coutint_s\np{\va{\State}_{s},\va{\Control}_{s},\va{\Uncertain}_{s+1}} 
                          \UppPlus \coutfin\np{ \va{\State}_{\horizon} } } 
                          \eqsepv t=\horizon-1,\ldots,0 \eqfinv
  \end{align}
\label{eq:Bellman_functions}
\end{subequations}
where \( \va{\State}_{t}=\state \in \STATE \),
\( \va{\State}_{s+1}=
\dynamics_{s}\np{\va{\State}_{s},\va{\Control}_{s},\va{\Uncertain}_{s+1}} \)
and \( \sigma\np{\va{\Control}_{s}} \subset \sigma\np{ \va{\State}_{s} } \),
for $s=t,\ldots,\horizon-1$.
In addition, the Bellman functions are nonnegative.                          %                           

\begin{assumption}
  We suppose that the Bellman functions in~\eqref{eq:Bellman_functions}
are measurable and satisfy the backward \emph{Bellman equation}
  \begin{equation}
     % \Value_\horizon= \coutfin \mtext{and}
    \Value_t\np{\state} =
    \inf_{\control \in \CONTROL} 
    \EE\bc{ \coutint_t(\state,\control,\va{\Uncertain}_{t+1}) \UppPlus  
      \Value_{t+1}\bp{ \dynamics_t\np{\state,\control,\va{\Uncertain}_{t+1}}}} 
    \eqsepv t=\horizon-1,\ldots,0 \eqfinp
    \label{eq:Bellman equation}
  \end{equation}
  This is the case when the 
  noise process~$\sequence{ \va{\Uncertain}_{t} }{t=1,\ldots,\horizon}$
  is a white noise and under technical assumptions 
  \cite{Bertsekas-Shreve:1996,Carpentier-Chancelier-Cohen-DeLara:2015}.
\end{assumption}

\subsubsection{Fenchel conjugates of the Bellman functions}

% Fenchel conjugates of value functions have been studied in different contexts
% \cite{DeMeyer:1996,Cardaliaguet-Rainer:2009,Pennanen-Ari-Pekka:2018}.

Now, we  provide a ``Bellman-like'' 
equation for the Fenchel conjugates of the value functions
(see \cite{Pennanen-Ari-Pekka:2018} for related considerations).

\begin{proposition}%[Fenchel meets Bellman] 
  The Bellman functions in~\eqref{eq:Bellman_functions}
satisfy the backward equalities
  \begin{equation}
    \Value_t\np{\state} = \inf_{\va{\State}\in\PRIMALBIS} 
    \bgp{ \inf_{\control \in \CONTROL} \Bp{
        \SFM{\bp{-\Hamiltonian\np{\state,\control,\cdot} }}{(-\star)}\np{\va{\State} } 
      } \UppPlus  \EE\bc{\Value_{t+1}\np{\va{\State}}}} 
    \eqsepv \forall t=\horizon-1,\ldots,0 
%   \quad\text{and}\quad \Value_\horizon=\coutfin
    \eqfinv
    \label{eq:Bellman equation_Hamiltonian}
  \end{equation}
  where the \emph{Hamiltonian}~$\Hamiltonian$ is defined by
  \begin{multline}
    \Hamiltonian\np{\state,\control,\va{\Dualstate}} =
    \EE\bc{ \coutint_t(\state,\control,\va{\Uncertain}_{t+1}) \UppPlus 
\proscal{\dynamics_t(\state,\control,\va{\Uncertain}_{t+1})}{\va{\Dualstate}}
    } \eqsepv %\\ 
\forall \np{ \state,\control,\va{\Dualstate} } 
\in \PRIMAL\times\CONTROL\times\DUALBIS
    \eqfinp
    \label{eq:Hamiltonian_Bellman}
  \end{multline}
  Moreover, letting $\ba{\LFM{\Value_t}}_{t =0,1,\ldots,\horizon}$ 
  be the Fenchel conjugates of the Bellman functions, 
    %     for $t =0,1,\dots,\horizon$ and 
  we have, for all $\dualstate \in \DUALSTATE$,
  \begin{equation}
    \LFM{\Value_t}\np{\dualstate} \leq \inf_{\va{\Dualstate}}
    \bgp{ \sup_{\control \in \CONTROL} 
      \Bp{\SFM{\Hamiltonian\np{\cdot,\control,\va{\Dualstate}}}{\star}
        \np{\dualstate}}
      \UppPlus \EE\bc{ \LFM{\Value_{t+1}}\np{\va{\Dualstate}} } } 
    \eqsepv \forall t=\horizon-1,\ldots,0 
    \eqfinp
    \label{ineq:fmb}
  \end{equation}
  \label{propo-fmb}
\end{proposition}

\begin{proof}
In what follows, we will manipulate mathematical expectations of 
random variables that are either nonnegative 
(by item~\ref{it:assumption:Bellman2} in Assumption~\ref{assumption:Bellman}),
or nonpositive (by taking the opposite),
or integrable (by item~\ref{it:assumption:Bellman1} 
in Assumption~\ref{assumption:Bellman}, giving random variables
resulting from a scalar product between an element
of $\espaceL{p}$ and one of $\espaceL{q}$).
We will be careful to remain in the conditions where the usual rules
of algebra apply \cite{Loeve:1977}.

  By the Bellman equation~\eqref{eq:Bellman equation}, we have that 
  \begin{align*}
    \Value_t\np{\state} 
    &=
      \inf_{\va{\State}\in\PRIMALBIS, \control \in \CONTROL} 
      \EE\bc{ \coutint_t(\state,\control,\va{\Uncertain}_{t+1}) \UppPlus  \Value_{t+1}\np{\va{\State}}} \nonumber 
    \\
    & \phantom{\va{\State}\in\PRIMALBIS} \text{s.t.} \quad 
      \va{\State} = \dynamics_t(\state,\control,\va{\Uncertain}_{t+1})
      \nonumber 
    \\
    &= 
      \inf_{\va{\State}\in\PRIMALBIS, \control \in \CONTROL} \sup_{\va{\Dualstate}\in\DUALBIS} 
      \Bp{
      \EE\bc{ \coutint_t(\state,\control,\va{\Uncertain}_{t+1}) 
      \UppPlus  \Value_{t+1}\np{\va{\State}} \UppPlus 
      \proscal{ \dynamics_t(\state,\control,\va{\Uncertain}_{t+1}) 
      - \va{\State} }{\va{\Dualstate}}}
      }
      \intertext{ by using item~\ref{it:assumption:Bellman1} 
      in Assumption~\ref{assumption:Bellman} }       
      \nonumber 
    &= 
      \inf_{\va{\State}\in\PRIMALBIS, \control \in \CONTROL} \bgp{
      \sup_{\va{\Dualstate}\in\DUALBIS} \Bp{ 
      \EE\bc{ \coutint_t(\state,\control,\va{\Uncertain}_{t+1}) 
      \UppPlus  
      \proscal{ \dynamics_t(\state,\control,\va{\Uncertain}_{t+1}) 
      - \va{\State} }{\va{\Dualstate}}} } 
      \UppPlus \EE\bc{ \Value_{t+1}\np{\va{\State}} }  }
      \intertext{ by \eqref{eq:upper_addition_sup_constant} with 
      \(0 \leq \EE\bc{ \Value_{t+1}\np{\va{\State}} } \)
      since the Bellman functions are nonnegative, 
      and as $\EE$ and $\UppPlus$ commute because, 
      by Assumption~\ref{assumption:Bellman}, all terms inside the
      expectation~$\EE$ are either nonnegative or integrable }        
      \nonumber 
    &= 
      \inf_{\va{\State}\in\PRIMALBIS, \control \in \CONTROL} 
      \Bgp{  
      \sup_{\va{\Dualstate}\in\DUALBIS}     \Bp{
      \EE\Bc{\proscal{- \va{\State}}{ \va{\Dualstate}}}  \UppPlus 
      \Hamiltonian\np{\state,\control,\va{\Dualstate}} } \UppPlus 
      \EE\bc{\Value_{t+1}\np{\va{\State}} } } 
      \nonumber 
      \intertext{ by definition~\eqref{eq:Hamiltonian_Bellman} 
      of the Hamiltonian }
    &=
      \inf_{\va{\State}\in\PRIMALBIS, \control \in \CONTROL} 
      \Bgp{  
      \sup_{\va{\Dualstate}\in\DUALBIS}  \Bp{
      \EE\Bc{\proscal{- \va{\State}}{ \va{\Dualstate}}}  \LowPlus 
      \Hamiltonian\np{\state,\control,\va{\Dualstate} } } \UppPlus 
      \EE\bc{\Value_{t+1}\np{\va{\State}} } }
      \nonumber 
      \intertext{ as Moreau upper and lower additions coincide 
above because \( -\infty < 
      \EE\Bc{\proscal{- \va{\State}}{ \va{\Dualstate}}} < +\infty \)
      by definition of the spaces $\PRIMALBIS$ and $\DUALBIS$ 
      and of the coupling~\eqref{eq:coupling:Bellman} between them}
    &= 
      \inf_{\va{\State}\in\PRIMALBIS, \control \in \CONTROL}
      \Bgp{  
      \sup_{\va{\Dualstate}\in\DUALBIS}   \bgp{ 
      \EE\Bc{\proscal{- \va{\State}}{ \va{\Dualstate}}}  \LowPlus 
      \Bp{- \bp{- \Hamiltonian\np{\state,\control,\va{\Dualstate} } } } } 
    \UppPlus 
      \EE\bc{\Value_{t+1}\np{\va{\State}} } }
      \nonumber 
    \\
    &= 
      \inf_{\va{\State}\in\PRIMALBIS, \control \in \CONTROL} 
      \Bp{
      \SFM{\bp{-\Hamiltonian\np{\state,\control,\cdot} }}{(-\star)}
      \np{\va{\State} } 
      \UppPlus 
      \EE\bc{\Value_{t+1}\np{\va{\State}}}
      }
      \nonumber
      \intertext{ by definition of the Fenchel conjugate of
      \( \va{\Dualstate} \mapsto -\Hamiltonian\np{\state,\control,\cdot} \) 
      with respect to the opposite coupling $\;^{(-\star)}$ defined by 
       \( \np{\va{\State},\va{\Dualstate}} \mapsto 
       \EE\bc{ \proscal{- \va{\State}}{ \va{\Dualstate} } } \),
      so that we have proven~\eqref{eq:Bellman equation_Hamiltonian} }
    &= 
      \inf_{\va{\State}\in\PRIMALBIS} 
      \bgp{ \inf_{\control \in \CONTROL} \Bp{
      \SFM{\bp{-\Hamiltonian\np{\state,\control,\cdot} }}{(-\star)}
      \np{\va{\State} } 
      } \UppPlus  \EE\bc{\Value_{t+1}\np{\va{\State}}}} 
      \intertext{ by the property~\eqref{eq:upper_addition_inf} that the 
operator~$\inf$ is ``distributive'' with respect to~$\UppPlus$ }
    &= 
      \inf_{\va{\State}\in\PRIMALBIS} 
      \bgp{ \inf_{\control \in \CONTROL} 
      \kernel_{\control}\np{ \state,\va{\State} } \UppPlus 
      \EE\bc{\Value_{t+1}\np{\va{\State}}}} 
      \eqfinv
      \nonumber
  \end{align*}
  where we have defined
  \begin{equation}
    \kernel_{\control}\np{ \state,\va{\State} } 
    = \SFM{ \bp{- \Hamiltonian\np{\state,\control,\cdot} }}{(-\star)}
    \np{\va{\State} } 
    \eqsepv \forall \control \in \CONTROL 
    \eqsepv \forall \np{ \state,\va{\State} } \in \PRIMAL\times\PRIMALBIS
    \eqfinp
  \end{equation}
  By~\eqref{eq:partial_Fenchel-Moreau_conjugate_coupling_kernel},
  we obtain that 
  \begin{equation}
    \SFM{ \kernel_{\control} }{ \SumCoupling{\star}{(-\star)} }
\np{\dualstate,\va{\Dualstate}}
    \leq \SFM{ \Hamiltonian\np{\cdot,\control,\va{\Dualstate}} }{\star}
    \np{\dualstate}
    \eqsepv \forall \control \in \CONTROL 
    \eqsepv \forall \np{ \dualstate,\va{\Dualstate} } \in \DUAL\times\DUALBIS
    \eqfinp
    \label{eq:partial_Fenchel-Moreau_conjugate_coupling_kernel_control}
  \end{equation}
  Therefore, as we have just established that
  \( \Value_t\np{\state} =
  \inf_{\va{\State}\in\PRIMALBIS} 
  \bgp{ \inf_{\control \in \CONTROL} 
    \kernel_{\control}\np{ \state,\va{\State} } 
    \UppPlus \EE\bc{\Value_{t+1}\np{\va{\State}}}} \),
  we deduce from implication~\eqref{eq:main} that 
  \begin{align*}
    \LFM{\Value_t}\np{\dualstate} 
    & \leq
      \inf_{\va{\Dualstate}}
      \bgp{  
      \SFM{ \bp{ \inf_{\control \in \CONTROL} \kernel_{\control} } }{\SumCoupling{\star}{(-\star)}}
      \np{ \dualstate,\va{\Dualstate} }
      \UppPlus \EE\bc{\Value_{t+1}\np{\cdot}}^{\star}(\va{\Dualstate})
      }\\
    & =
      \inf_{\va{\Dualstate}}
      \bgp{  \sup_{\control \in \CONTROL} \Bp{ 
      \SFM{ \kernel_{\control} }{\SumCoupling{\star}{(-\star)}}
      \np{ \dualstate,\va{\Dualstate} } }
      \UppPlus \EE\bc{\Value_{t+1}\np{\cdot}}^{\star}(\va{\Dualstate})
      }
      \intertext{ since 
      \( \SFM{ \bp{ \inf_{\control \in \CONTROL} \kernel_{\control} } }{\SumCoupling{\star}{(-\star)}}
      = \sup_{\control \in \CONTROL} 
      \SFM{ \kernel_{\control} }{\SumCoupling{\star}{(-\star)}} \)
      by the formula~\eqref{eq:Fenchel-Moreau_conjugate_infimum} }
    & \leq 
      \inf_{\va{\Dualstate}}
      \Bp{ \sup_{\control \in \CONTROL} \Bp{ 
      \SFM{\Hamiltonian\np{\cdot,\control,\va{\Dualstate}}}{\star}
      \np{\dualstate} }
      \UppPlus \EE\bc{\Value_{t+1}\np{\cdot}}^{\star}(\va{\Dualstate})
      }
      \intertext{ as
      \( \SFM{ \kernel_{\control}}{\SumCoupling{\star}{(-\star)}}
      \np{\dualstate,\va{\Dualstate}} 
      \leq \SFM{ \Hamiltonian\np{\cdot,\control,\va{\Dualstate} } }{\star} 
      \np{\dualstate} \) 
      by~\eqref{eq:partial_Fenchel-Moreau_conjugate_coupling_kernel_control}}
    & \leq
      \inf_{\va{\Dualstate}}
      \Bp{ \sup_{\control \in \CONTROL} \Bp{ 
      \SFM{\Hamiltonian\np{\cdot,\control,\va{\Dualstate}}}{\star}
      \np{\dualstate} }
      \UppPlus \EE\bc{ \LFM{\Value_{t+1}}\np{\va{\Dualstate}} } } \eqfinv
  \end{align*}
as soon as we prove that 
\( \EE\bc{\Value_{t+1}\np{\cdot}}^{\star}(\va{\Dualstate})
\leq \EE\bc{ \LFM{\Value_{t+1}}\np{\va{\Dualstate}} } \).

Indeed, we have that
  \begin{align*}
    \EE\bc{\Value_{t+1}\np{\cdot}}^{\star}(\va{\Dualstate}) 
    &=
      \sup_{\va{\State}\in\PRIMALBIS} \Bp{ 
      \EE\Bc{ \proscal{\va{\State}}{ \va{\Dualstate} } }  \LowPlus \bp{- 
      \EE\bc{ \Value_{t+1}\np{\va{\State}} } } } \\
    & =
      \sup_{\va{\State}\in\PRIMALBIS} \Bp{ 
      \EE\Bc{ \proscal{\va{\State}}{ \va{\Dualstate} } \LowPlus \bp{- 
      \Value_{t+1}\np{\va{\State}} } } } 
      \intertext{ because \( -\infty < 
      \EE\Bc{\proscal{- \va{\State}}{ \va{\Dualstate}}} < +\infty \)
      by definition of the spaces $\PRIMALBIS$ and $\DUALBIS$ 
      and of the coupling~\eqref{eq:coupling:Bellman} between them}
    & \leq 
      \EE\Bc{ \sup_{\va{\State}\in\PRIMALBIS} \Bp{ 
      \proscal{\va{\State}}{ \va{\Dualstate} } \LowPlus \bp{- 
      \Value_{t+1}\np{\va{\State}} } } } =  
      \EE\bc{ \LFM{\Value_{t+1}}\np{\va{\Dualstate}} } 
      \eqfinv
  \end{align*}
by definition of \( \LFM{\Value_{t+1}}\np{\va{\Dualstate}} \).
  This ends the proof.
\end{proof}

Proposition~\ref{propo-fmb} may be useful to obtain upper and lower
estimates in approximations of Bellman functions.
We just provide a sketch of the argument.
\begin{enumerate}
\item 
Suppose that the Bellman functions 
 $\sequence{\Value_t}{t =0,1,\ldots,\horizon}$ 
% in~\eqref{eq:bellman} 
satisfy the Bellman equation~\eqref{eq:Bellman equation}
and are convex l.s.c.. 
This is the case in 
\emph{Stochastic Dual Dynamic Programming} (SDDP), when 
the dynamics~\( \dynamics_{t} \) are jointly linear in state and control, 
the instantaneous costs~\( \coutint_{t} \) are 
jointly convex in state and control,
the final cost~\( \coutfin \) is convex, 
together with technical assumptions
(see details in \cite{MR2746854,girardeau2014convergence} 
and references therein).
\item 
The Fenchel conjugates $\sequence{\LFM{\Value_t}}{t =0,1,\ldots,\horizon}$ 
of the Bellman functions are convex l.s.c., by construction.
Suppose that they satisfy a ``Bellman like'' \emph{equation}
  \begin{equation}
    \LFM{\Value_t}\np{\dualstate} = \inf_{\va{\Dualstate}}
    \bgp{ \sup_{\control \in \CONTROL} 
      \Bp{\SFM{\Hamiltonian\np{\cdot,\control,\va{\Dualstate}}}{\star}
        \np{\dualstate}}
      \UppPlus \EE\bc{ \LFM{\Value_{t+1}}\np{\va{\Dualstate}} } } 
    \eqsepv \forall t=\horizon-1,\ldots,0 \eqfinv 
    \label{eq:fmb}
  \end{equation}
which is~\eqref{ineq:fmb}, where the inequality is an equality.
For this, one needs assumptions of the kind described 
in~\S\ref{The_duality_equality_case}, as well as the equality 
 $ \EE\bc{\Value_{t+1}\np{\cdot}}^{\star}(\va{\Dualstate})
    = \EE\bc{ \LFM{\Value_{t+1}}\np{\va{\Dualstate}} }$ 
(see~\cite{rockafellar1968integrals,Rockafellar1971integrals}).
\item 
With the Bellman operators deduced from the 
Bellman equation~\eqref{eq:Bellman equation}
and ``Bellman like'' equation~\eqref{ineq:fmb},
one can easily produce piecewise linear lower bound functions
\( \underline{\Value}_{t, (k)} \leq \underline{\Value}_{t, (k+1)}
 \leq {\Value}_{t} \) and
\( \underline{\tilde\Value}_{t, (k)} \leq 
\underline{\tilde\Value}_{t, (k+1)} \leq \LFM{{\Value}_{t}} \),
for \(k\in\NN\), 
by a proper algorithm (like the SDDP algorithm).
\item 
Since the Bellman functions 
 $\sequence{\Value_t}{t =0,1,\ldots,\horizon}$ 
are convex l.s.c. and proper\footnote{%
They are nonnegative and we exclude the degenerate case where they would have an empty domain.}, we deduce that
\begin{equation}
 \underline{\Value}_{t, (k)} \leq \underline{\Value}_{t, (k+1)}
 \leq {\Value}_{t} = \LFMbi{{\Value}_{t}} \leq 
\LFM{ \underline{\tilde\Value}_{t, (k+1)} } \leq
\LFM{ \underline{\tilde\Value}_{t, (k)} } \eqfinp
\end{equation}
Thus, we can control the evolution of the algorithm.
\end{enumerate}

\section{Conclusion}

Working on Fenchel conjugates of Bellman functions,
we obtained Inequality~\eqref{ineq:fmb}.
This was our impetus to investigate 
inequalities with more than one coupling.
The sum coupling \( \SumCoupling{\coupling}{\couplingbis} \) 
in~\eqref{eq:product_coupling} leads to other nice formulas like
the following one (with obvious notations)
\begin{equation}
  \fonctiontrois\np{\primal,\primalbis} =
\inf_{\control \in \CONTROL} \Bp{ {\cal L} \np{\primal,\control}
\UppPlus {\cal M} \np{\control,\primalbis} } 
\Rightarrow 
\SFM{\fonctiontrois}{\SumCoupling{\coupling}{\couplingbis}}\np{\dual,\dualbis} 
= \sup_{\control \in \CONTROL} \Bp{ 
\SFM{{\cal L}}{\coupling}\np{\dual,\control} 
\LowPlus
\SFM{{\cal M}}{\couplingbis}\np{\control,\dualbis} }
\eqfinv
\end{equation}
where the conjugates \( \SFM{{\cal L}}{\coupling} \) and
\( \SFM{{\cal M}}{\couplingbis} \) are partial.
\bigskip

\textbf{Acknowledgements.}
We want to thank 
%Antonin Chambolle,
Juan Enrique Martínez Legaz,
Roger Wets, 
Lionel Thibault,
Terry Rockafellar,
Guillaume Carlier,
Filippo Santambrogio,
Jean-Paul Penot,
Patrick L. Combettes,
Marianne Akian, 
Stéphane Gaubert,
Michel Théra,
Michel Volle,
for their comments on first versions of this work. 

\appendix 

\section{Appendix}
\label{Appendix}

\subsection{Background on J.~J. Moreau lower and upper additions}
\label{Moreau_lower_and_upper_additions}

% \subsection{Background on J.~J. Moreau lower and upper additions}
% \label{Moreau_lower_and_upper_additions}

% A function $f: \Primal \to \bar\RR$, where 
% \( \bar\RR = [-\infty,+\infty] \) is said to be an
% \emph{extended function}, that we will shorten into function.

% The \emph{domain}~$\dom f$ of the extended function 
% $f: \Primal \to ]-\infty,+\infty] $ is
% \begin{equation}
% \dom f = \{ \primal \in \Primal \mid f(\primal) < +\infty \} \eqfinp 
% \end{equation}
% The extended function $f: \Primal \to ]-\infty,+\infty] $ 
% is said to be \emph{proper} if $\dom f \not = \emptyset$. 
% Thus, a \emph{proper function} cannot take the value $-\infty$,
% and must take at least one finite value. 

When we manipulate functions with values 
in~$\bar\RR = [-\infty,+\infty] $,
we adopt the following Moreau \emph{lower addition} or
\emph{upper addition}, depending on whether we deal with $\sup$ or $\inf$
operations. 
We follow \cite{Moreau:1970}.
In the sequel, $u$, $v$ and $w$ are any elements of~$\bar\RR$.

\subsubsection*{Moreau lower addition}

\begin{subequations}
  The Moreau \emph{lower addition} extends the usual addition with 
  \begin{equation}
    % \np{+\infty} - \np{+\infty} = \np{-\infty} - \np{-\infty} = -\infty \eqfinp
    \np{+\infty} \LowPlus \np{-\infty} = \np{-\infty} \LowPlus \np{+\infty} = -\infty \eqfinp
    \label{eq:lower_addition}
  \end{equation}
  % From the lower addition, we naturally deduce the \emph{lower substraction}:
  % \begin{equation}
  %   u \minusdot v = u \LowPlus \np{-v}  \eqfinp
  %   \label{eq:lower_substraction}
  % \end{equation}
  With the \emph{lower addition}, \( \np{\bar\RR, \LowPlus } \) is a convex cone,
  with $ \LowPlus $ commutative and associative.
  The lower addition displays the following properties:
  \begin{align}
    u \leq u' \eqsepv v \leq v' 
    & 
      \Rightarrow u \LowPlus v \leq u'  \LowPlus v' \eqfinv \label{eq:lower_addition_leq}\\
    (-u) \LowPlus (-v) 
    & 
      \leq -(u \LowPlus v) \eqfinv \label{eq:lower_addition_minus}\\ 
    (-u) \LowPlus u 
    & 
      \leq 0 \eqfinv \label{eq:lower_substraction_le_zero}\\
    \sup_{a\in\AAA} f(a) \LowPlus \sup_{b\in\BB} g(b)
    &=
      \sup_{a\in\AAA,b\in\BB} \bp{f(a) \LowPlus g(b)} \eqfinv\label{eq:lower_addition_sup}\\ 
    \inf_{a\in\AAA} f(a) \LowPlus \inf_{b\in\BB} g(b) 
    & \leq 
      \inf_{a\in\AAA,b\in\BB} \bp{f(a) \LowPlus g(b)} \eqfinv\label{eq:lower_addition_inf}\\ 
    t < +\infty \Rightarrow    \inf_{a\in\AAA} f(a) \LowPlus t 
    & =
      \inf_{a\in\AAA} \bp{f(a) \LowPlus t} \eqfinp\label{eq:lower_addition_inf_constant}
  \end{align}
\end{subequations}

\subsubsection*{Moreau upper addition}

\begin{subequations}
  The Moreau \emph{upper addition} extends the usual addition with 
  \begin{equation}
    \np{+\infty} \UppPlus \np{-\infty} = 
    \np{-\infty} \UppPlus \np{+\infty} = +\infty \eqfinp
    \label{eq:upper_addition}
  \end{equation}
  % From the upper addition, we naturally deduce the \emph{upper substraction}:
  % \begin{equation}
  %   u \dotminus v = u \UppPlus \np{-v}  \eqfinp
  %   \label{eq:upper_substraction}
  % \end{equation}
  With the \emph{upper addition}, \( \np{\bar\RR, \UppPlus } \) is a convex cone,
  with $ \UppPlus $ commutative and associative.
  The upper addition displays the following properties:
  \begin{align}
    u \leq u' \eqsepv v \leq v' 
    & 
      \Rightarrow u \UppPlus v \leq u' \UppPlus v' \eqfinv
      \label{eq:upper_addition_leq} \\
    (-u) \UppPlus (-v) 
    & 
      \geq -(u \UppPlus v) \eqfinv
      \label{eq:upper_addition_minus} 
    \\ 
    (-u) \UppPlus u 
    & 
      \geq 0 \eqfinv
      \label{eq:upper_substraction_ge_zero} 
    \\ 
    \inf_{a\in\AAA} f(a) \UppPlus \inf_{b\in\BB} g(b) 
    & =
      \inf_{a\in\AAA,b\in\BB} \bp{f(a) \UppPlus g(b)} \eqfinv
      \label{eq:upper_addition_inf} 
    \\
    \sup_{a\in\AAA} f(a) \UppPlus \sup_{b\in\BB} g(b)
    & \geq 
      \sup_{a\in\AAA,b\in\BB} \bp{f(a) \UppPlus g(b)} \eqfinv
      \label{eq:upper_addition_sup}
    \\ 
    -\infty < t \Rightarrow    \sup_{a\in\AAA} f(a) \UppPlus t 
    & =
      \sup_{a\in\AAA} \bp{f(a) \UppPlus t} \eqfinp
      \label{eq:upper_addition_sup_constant}
      % \\  u \UppPlus (v \LowPlus w) & \geq  (u \UppPlus v) \LowPlus w  
  \end{align}
\end{subequations}

\subsubsection*{Joint properties of the Moreau lower and upper addition}

\begin{subequations}
We obviously have that
\begin{equation}
  u \LowPlus v \leq u \UppPlus v \eqfinp 
\label{eq:lower_leq_upper_addition}
\end{equation}
The Moreau lower and upper additions are related by
\begin{equation}
-(u \UppPlus v) = (-u) \LowPlus (-v) \eqsepv 
-(u \LowPlus v) = (-u) \UppPlus (-v) \eqfinp 
\label{eq:lower_upper_addition_minus}  
\end{equation}
They satisfy the inequality
\begin{equation}
(u \UppPlus v) \LowPlus w \leq u \UppPlus (v \LowPlus w) \eqfinp
                                 \label{eq:lower_upper_addition_inequality} 
\end{equation}
with 
  \begin{equation}
(u \UppPlus v) \LowPlus w < u \UppPlus (v \LowPlus w) 
\iff
\begin{cases}
 u=+\infty \mtext{ and } w=-\infty \eqsepv \\
\mtext{ or } \\
u=-\infty \mtext{ and } w=+\infty 
\mtext{ and } -\infty < v < +\infty \eqfinp
\end{cases}
                                 \label{eq:lower_upper_addition_equality} 
\end{equation}
Finally, we have that 
\begin{align}
& u \LowPlus (-v) \leq 0 \iff u \leq v  \iff  0 \leq v \UppPlus (-u) 
\eqfinv
\label{eq:lower_upper_addition_comparisons}
\\
& u \LowPlus (-v) \leq w \iff u \leq v \UppPlus w \iff u \LowPlus (-w) \leq v
\eqfinv 
\\  
& w \leq v \UppPlus (-u) \iff u \LowPlus w \leq v \iff u \leq v \UppPlus (-w) 
\eqfinp
\end{align}
 \end{subequations}

% \begin{lemma} $a \LowPlus ( - b) \le 0 \iff a \le b  \iff  (-a) \UppPlus b \ge 0$
% \label{lem:comparaison}
% \end{lemma}

% \begin{proof}
%   We prove the first equivalence and leave the second one to the reader. 
%   Assume that $a \LowPlus ( - b) \le 0$.
%   Now, to obtain $a \le b$, we consider three cases.
%   \begin{enumerate}
%   \item 
%     If \( -\infty < b < +\infty \), then from 
%     \( a \LowPlus (-b)\le 0 \), we deduce that $a$ cannot take the value $+ \infty$. If $a$ is finite we have 
%     $a \le b$ by standard algebra and if $a$ is equal to $-\infty$ we indeed also have $a = -\infty \le b$.
%   \item 
%     If \( b =-\infty \), then from 
%     \( a \LowPlus \bp{ - b } \leq 0 \), we deduce that 
%     \( a \LowPlus (+\infty) \leq 0 \), hence that $a = -\infty$. Therefore, 
%     \( a =-\infty \leq b =-\infty \).
%   \item 
%     If \( b = +\infty \), then we have \( a \leq +\infty = b \).
%   \end{enumerate}
%   Assume that $a \le b$, we then have by Equation~\eqref{eq:lower_addition_leq} that 
%   $a \LowPlus (-b) \le b \LowPlus (-b)$. Which gives $a \LowPlus (-b) \le b \LowPlus (-b) \le 0$ 
%   using Equation~\eqref{eq:lower_substraction_le_zero}. 
% \end{proof}

\subsection{Background on Fenchel-Moreau conjugacy with respect to a coupling}

Let be given two sets $\PRIMAL$ and $\DUAL$.
Consider a \emph{coupling} function 
\( \coupling : \PRIMAL \times \DUAL \to \barRR \).
% \( \coupling : \PRIMAL \times \DUAL \to  ]-\infty,+\infty[ \).
We also use the notation \( \PRIMAL 
\overset{\coupling}{\leftrightarrow} \DUAL \) for a coupling, so that
\begin{equation}
  \PRIMAL \overset{\coupling}{\leftrightarrow} \DUAL 
  \iff
  \coupling : \PRIMAL \times \DUAL \to \barRR \eqfinp
  \label{eq:coupling}
\end{equation}

\begin{definition}
  The \emph{Fenchel-Moreau conjugate} of a 
  function \( \fonctionprimal : \PRIMAL  \to \barRR \), 
  with respect to the coupling~$\coupling$ in~\eqref{eq:coupling}, is
  the function \( \SFM{\fonctionprimal}{\coupling} : \DUAL  \to \barRR \) 
  defined by
  \begin{equation}
    \SFM{\fonctionprimal}{\coupling}\np{\dual} = 
    \sup_{\primal \in \PRIMAL} \Bp{ \coupling\np{\primal,\dual} 
      \LowPlus \bp{ -\fonctionprimal\np{\primal} } } 
    \eqsepv \forall \dual \in \DUAL
    \eqfinp
    \label{eq:upper-Fenchel-Moreau_conjugate}
  \end{equation}
  We associate with the coupling $\coupling$
the coupling $\coupling': \DUAL \times \PRIMAL \to \barRR$ defined by 
  $\coupling'\np{\dual,\primal}= \coupling\np{\primal,\dual}$.
  The \emph{Fenchel-Moreau biconjugate} is
  the function \( \SFMbi{\fonctionprimal}{\coupling} : \PRIMAL  \to \barRR \) defined by
  \begin{equation}
    \SFMbi{\fonctionprimal}{\coupling}\np{\primal} 
    = \np{\fonctionprimal^{\coupling}}^{\coupling'}\np{\primal} 
    = \sup_{\dual \in \DUAL} 
    \Bp{ \coupling\np{\primal,\dual} 
      \LowPlus \bp{ - \SFM{\fonctionprimal}{\coupling}\np{\dual} } } 
    \eqsepv \forall \primal \in \PRIMAL 
    \eqfinp
    \label{eq:Fenchel-Moreau_biconjugate}
  \end{equation}
\end{definition}

The following property is well known. 
\begin{proposition}
  For any function \( \fonctionprimal : \PRIMAL  \to \barRR \), we have that 
  \begin{equation}
    \SFMbi{\fonctionprimal}{\coupling}\np{\primal} \leq \fonctionprimal\np{\primal}  \eqfinp  
    \label{eq:Fenchel-Moreau_biconjugate_inequality} 
  \end{equation}
  % \label{pr:condition_Fenchel-Moreau_biconjugate}
\end{proposition}

\begin{proof}
  We prove~\eqref{eq:Fenchel-Moreau_biconjugate_inequality} as follows.
  \begin{subequations}
    \begin{align}
      \SFMbi{\fonctionprimal}{\coupling}\np{\primal} 
      \LowPlus \bp{ -\fonctionprimal\np{\primal} }
      &= \sup_{\dual \in \DUAL} 
        \Bp{ \coupling\np{\primal,\dual} \LowPlus \bp{-\SFM{\fonctionprimal}{\coupling}\np{\dual} } }
        \LowPlus \bp{-\fonctionprimal\np{\primal}}
        \tag{by~\eqref{eq:Fenchel-Moreau_biconjugate} 
        and~\eqref{eq:lower_addition_sup} }
      \\
      &= \sup_{\dual \in \DUAL} 
        \bgp{ \Bp{ \coupling\np{\primal,\dual} \LowPlus 
        \bp{-\SFM{\fonctionprimal}{\coupling}\np{\dual} } } 
        \LowPlus \bp{-\fonctionprimal\np{\primal}} } 
        \tag{by~\eqref{eq:lower_addition_sup} } 
      \\
      &= \sup_{\dual \in \DUAL} 
        \bgp{ \coupling\np{\primal,\dual} \LowPlus 
        \bp{-\SFM{\fonctionprimal}{\coupling}\np{\dual} } 
        \LowPlus \bp{-\fonctionprimal\np{\primal}} } 
        \tag{by associativity of~$\LowPlus$}
      \\
      &= \sup_{\dual \in \DUAL} 
        \bgp{ \coupling\np{\primal,\dual} \LowPlus \bp{-\fonctionprimal\np{\primal}} 
        \LowPlus \bp{-\SFM{\fonctionprimal}{\coupling}\np{\dual} } 
        } 
        \tag{by commutativity of~$\LowPlus$}
      \\
      & \leq \sup_{\dual \in \DUAL} 
        \bgp{
        \sup_{\primal \in \PRIMAL} \Bp{ \coupling\np{\primal,\dual} \LowPlus \bp{-\fonctionprimal\np{\primal}} }
        \LowPlus \bp{-\SFM{\fonctionprimal}{\coupling}\np{\dual} } 
        } 
        \tag{by~\eqref{eq:lower_addition_leq} }
      \\
      &= \sup_{\dual \in \DUAL} 
        \Bp{ \SFM{\fonctionprimal}{\coupling}\np{\dual} 
        \LowPlus \bp{-\SFM{\fonctionprimal}{\coupling}\np{\dual} } } 
        \tag{by~\eqref{eq:upper-Fenchel-Moreau_conjugate} }
      \\
      & \leq 0 \eqfinp 
        \tag{by~\eqref{eq:lower_substraction_le_zero}}
    \end{align}
  \end{subequations}
  We have obtained that \( \SFMbi{\fonctionprimal}{\coupling}\np{\primal} 
  \LowPlus \bp{ -\fonctionprimal\np{\primal} } \leq 0 \). 
  Now, using~\eqref{eq:lower_upper_addition_comparisons},
  we obtain~\eqref{eq:Fenchel-Moreau_biconjugate_inequality}. 
  This ends the proof.
\end{proof}

The following properties are easy to establish.
\begin{proposition}
  For any family 
  \( \sequence{\fonctionprimal_{\control}}{\control\in\CONTROL} \) of
  functions \( \fonctionprimal_{\control} : \PRIMAL  \to \barRR \), we have that
  \begin{subequations}
    \begin{align}
      \SFM{\bp{\inf_{\control\in\CONTROL}\fonctionprimal_{\control}}}{\coupling}
\np{\primal}  
      &= \sup_{\control\in\CONTROL}\SFM{\fonctionprimal_{\control}}{\coupling}
\np{\primal} 
      \\
      - \SFM{\bp{\inf_{\control\in\CONTROL}\fonctionprimal_{\control}}}{\coupling}
\np{\primal}  
      &= \inf_{\control\in\CONTROL} \bp{-\SFM{\fonctionprimal_{\control}}{\coupling}
\np{\primal} }
        \eqfinp
        \label{eq:Fenchel-Moreau_conjugate_infimum}
    \end{align}
  \end{subequations}
\end{proposition}

\begin{proposition}
Let be given two ``primal'' sets $\PRIMAL$, $\PRIMALBIS$
and two ``dual'' sets $\DUAL$, $\DUALBIS$, 
together   with two coupling functions
  \begin{equation}
    \coupling : \PRIMAL \times \DUAL \to \barRR \eqsepv
    \couplingbis : \PRIMALBIS \times \DUALBIS \to \barRR \eqfinp 
  \end{equation}
  For any bivariate function 
  \( \kernel : \PRIMAL \times \PRIMALBIS \to \barRR \),
  we have that 
  \begin{equation}
    \SFM{ \Bp{ - \bp{ \primal \mapsto 
          \SFM{\kernel\np{\primal,\cdot}}{\couplingbis} }}}{\coupling}
    = \SFM{\kernel}{\coupling \LowPlus \couplingbis} 
    =   \SFM{ \Bp{ - \bp{ \primalbis \mapsto 
          \SFM{\kernel\np{\cdot,\primalbis}}{\coupling} }}}{\couplingbis}
    \eqfinp
\label{eq:composition_of_conjugates}
  \end{equation}
\end{proposition}

\newcommand{\noopsort}[1]{} \ifx\undefined\allcaps\def\allcaps#1{#1}\fi

\end{document}